\documentclass[11pt,reqno]{amsart}
\usepackage{amsmath, amsthm, amsfonts, amssymb, color}
\usepackage{mathrsfs, bbm}

\allowdisplaybreaks

\newtheorem{thm}{Theorem}[section]
\newtheorem{cor}[thm]{Corollary}
\newtheorem{lem}[thm]{Lemma}
\newtheorem{remark}[thm]{Remark}
\newtheorem{prp}[thm]{Proposition}

\newtheorem{exa}[thm]{Example}
\theoremstyle{definition}
\newtheorem{defn}[thm]{Definition}
\newcommand{\scr}[1]{\mathscr #1}
\definecolor{wco}{rgb}{0.5,0.2,0.3}

\numberwithin{equation}{section}

\def\1{{\mathbbm 1}}

\def\R{\mathbb R}

\def\E{\mathbb E}

\def\sL{\mathcal L}
\def\L{\mathcal L}
\def\S{\mathcal S}

\def\F{\mathcal F}

\def\wh{\widehat}

\def\<{\langle}
\def\>{\rangle}

\def\aa{\alpha}

\def\D{\scr D}

\def\pf{\noindent{\bf Proof.} }

 \def\beq{\begin{equation}}

 \def\P{\mathbb P} 
  \def\ee{\varepsilon}

\def\sX{\mathcal{X}}

 \makeatletter
\@namedef{subjclassname@2020}{%
  \textup{2020} Mathematics Subject Classification}
\makeatother

\begin{document}
\bibliographystyle{plain}

\title{ Boundary Harnack principle for diffusion with jumps  on  metric measure spaces }

\author{Zhen-Qing Chen}
\address{Department of Mathematics, University of Washington, Seattle, WA 98195, USA}
\email{zqchen@uw.edu}
\thanks{}
\author{Jie-Ming Wang}
\address{School of Mathematics and Statistics,  Beijing Institute of Technology,
Beijing 100081, P. R. China}
\curraddr{}
\email{wangjm@bit.edu.cn}
\thanks{}



\date{}

\begin{abstract}
A non-scale invariant BHP  on any open set is obtained for a   large
class of discontinuous Hunt processes on metric measure spaces that are in weak duality with another Hunt process under
 suitable conditions
  in terms of  estimates of exit distributions   from open sets, and bounds on Green functions
  and  the comparability of the  jump density function.
  These conditions are easy to verify in concrete cases.
 Several  examples are given to illustrate the main results and the new contributions of this paper.
  Our result in particular establishes the BHP on any open set  for
 any Hunt process associated with a regular symmetric Dirichlet form
   on a metric measure space having both the strongly local term and  the pure-jump term that admits a two-sided  heat kernel estimates of the mixture of (sub-)Gaussian and stable-like form, as well as for a wide class of non-symmetric diffusion processes with jumps on  $\R^d.$
\end{abstract}

\maketitle

\section{Introduction}
\subsection{Background and motivation}

Let $\sL$ be an operator  on a locally compact separable  metric space $(\sX,d)$ that generates a Hunt process $X$.
  For each $x\in \sX$ and $r>0$, denote by  $B(x,r):=\{y\in \mathcal{X}:\,d(x,y)<r\}$   the open ball in $\sX$ centered at $x$ with radius $r$.
  Here and below, we use $:=$ as a way of definition.
  We say the boundary Harnack principle (BHP in abbreviation)  for $\sL$  (or, equivalently, for the Hunt process $X$)   holds  on an open subset $D\subset\sX$ if   for any $z\in\partial D$ and $r\in (0, R)$ for some $R>0$, there is a constant  $C=C(z, r, D)>1$   so that
\begin{equation}\label{e:1.0}
h_1(x)h_2(y)\leq C\,h_1(y)h_2(x)\quad\hbox{ for every }x,y\in D\cap B(z,r)
\end{equation}
for any non-negative functions $h_1$ and $h_2$  on $\sX$ that are $\sL$-harmonic in $D\cap B(z,2r)$ and vanish on $B(z,2r)\setminus D$ in a suitable sense. If there is  some $R \in (0, \infty]$ so that the above holds for some constant $C>1$ uniformly in  $z \in \partial D$ and in $r\in (0, R)$,  we say the scale invariant BHP holds for $X$ and $\sL$ with radius $R$ on $D$. If the constant $C>1$ can be further taken to be independent of the open subset $D\subset \sX$ , we say the uniform BHP holds for $X$ and $\sL$ with radius $R$.  Whether BHP holds or not is a fundamental problem both in  analysis and in probability theory.
 While the scale-invariant and uniform BHPs have the most important consequences and applications,
the non-scale invariant BHP   remains highly useful because  it ensures that two positive harmonic functions vanishing
on a portion of the boundary decay at the same rate.
For example, the non-scale invariant BHP is used in the study of Martin boundary  for
 uniformly elliptic operators in non-Lipschitz domains in \cite{BB3}.
  The non-scale invariant BHP in subdomains of $\R^3$ with the trace of three dimensional Brownian motion
 removed is used to study  analyticity of three dimensional Brownian intersection exponents
 in Brownian slit domain  \cite{GLLLL}. In both examples,  the scale-invariant BHP  may fail.

The scale invariant BHP  for the Laplace operator $\Delta$ (or equivalently, for Brownian motion)
 in Lipschitz domains in $\R^d$  was  independently established  by
  Ancona \cite{A},  Dahlberg \cite{Da} and Wu  \cite{Wu}.
  This BHP  result
   was    later extended
 to  non-tangential accessible (NTA) domains by Jerison and Kenig \cite{JK}, to uniform domains by Aikawa \cite{Ai}
and to more general  elliptic operators in Lipschitz domains  by Caffarelli, Fabes, Mortola and Salsa
 \cite{CFMS} and Fabes, Garofalo, Marin-Malave
and Salsa \cite{FGMS}.
 Aikawa \cite{Ai2} further showed for Laplace operator $\Delta$ that the scale-invariant
BHP holds on $D\subset \R^d$ if and only if $D$ is an inner uniform domain.
Bass and Burdzy \cite{BB1} developed a probabilistic approach  to give an alternative proof of non-scale invariant   BHP in Lipschitz domains
for Brownian motion.
This probabilistic method was further developed   to obtain a non-scale invariant BHP for uniformly elliptic  second order differential
operators   in H\"older domains and twisted H\"older domains in \cite{BB2, BBB, BB4}.
 In particular, it is shown in  \cite{BB2} that
 the   non-scale invariant   BHP holds for any twisted H\"older domain of order $\alpha$ with $\alpha \in (1/2, 1]$, while
 for each $\alpha \in (0, 1/2)$, there is some   twisted H\"older domain of order $\alpha$ on which
 the   non-scale invariant   BHP fails.
There is  now
 a very substantial literature on BHP for differential operators.
 For diffusions on geodesic spaces,  the scale invariant BHP on inner uniform domains  has
  been established in \cite{GSC, LS, Lie} under a  two-sided  Gaussian or sub-Gaussian heat kernel bounds assumption.
The results of \cite{BM} combining with \cite{BCM}  show that  for symmetric diffusions on geodesic spaces,  the scale invariant BHP on inner uniform domains holds under the assumption of scale invariant elliptic Harnack inequality.
 Recently,   this result
  has been further extended to the case that the underlying space may not  be geodesic in \cite{Cha}.

\smallskip

The study of BHP for non-local operators (or equivalently, discontinuous Markov processes)
started relatively recent and many  progresses have been made in recent years.
  Bogdan \cite{Bo} showed the scale invariant BHP holds for $\Delta^{\alpha/2}:=-(-\Delta)^{\alpha/2}$
(or equivalently, for rotationally symmetric $\alpha$-stable processes)
 with $\alpha\in (0,2)$ in bounded Lipschitz domains.
This result was later extended to $\kappa$-fat open sets by Song and Wu \cite{SW} and to the uniform BHP   in arbitrary open sets
by Bogdan, Kulczycki and Kwasnicki in \cite{BKK}.
The scale-invariant BHP  for   pure jump subordinate Brownian motions in
 $\kappa$-fat  (or corkscrew)
open sets and the uniform BHP for
a  class of pure jump symmetric L\'evy process in open sets have been shown in  \cite{KSV, KSV5}.
BHP has also been established for some non-homogenous non-local operators, see  \cite{BBC, KS, CRY}. A scale invariant BHP has been shown to hold on open sets for $s$-stable-subordination with $s\in (0, 1)$ of symmetric diffusion processes that  have
two-sided sub-Gaussian heat kernel estimates on  Sierpinski carpets and gaskets  in \cite{Sto, KK}
and on Ahlfors $n$-regular metric measure spaces   in \cite[Example 5.7]{BKK}.
  A non-scale-invariant BHP for stable-like operator in non-divergence form is given in \cite{ROS}.
Very recently, Cao and Chen \cite{CC} established a uniform BHP on any open sets for a large class of non-local operators on metric measure spaces  under a jump measure comparability and tail estimate condition, and an upper bound condition on the distribution function for the exit times from balls. These conditions are satisfied by any  non-local operator  that admits a two-sided mixed stable-like heat kernel bounds when  the underlying metric measure spaces have  volume doubling and reverse volume doubling properties.
The non-local operators studied in \cite{CC} can be   symmetric as well as  non-symmetric.

 Despite the significant progresses on pure non-local operators, the study on the BHP
 for discontinuous processes with Gaussian components
   (or equivalently,  for non-local operators with diffusive components)
 is still limited.
For discontinuous processes with Gaussian components,
   the scale invariant BHP   is  shown to hold on  $C^{1,1}$-smooth open sets for
subordinate Brownian motion with Gaussian components in \cite{CKSV,  KSV2},
 and for a large class of non-symmetric diffusions with jumps in $\R^d$ in \cite{CW1}.
  In our  recent paper \cite{CW2},
    it is shown that the scale invariant BHP   for subordinate Brownian motion with Gaussian components on $\R^d$    holds on Lipschitz domains satisfying an interior cone condition with common angle $\theta \in (\cos^{-1}(1/\sqrt{d}),\pi)$ but fails on truncated circular cones with angle $\theta \in (0, \cos^{-1}(1/\sqrt{d})]$.
   This is in  stark
    contrast with the case of differential operators  such as $\Delta$ and with the case of pure non-local operators
  such as $\Delta^{\alpha/2}$.
  However, non-scale invariant BHP for discontinuous processes with Gaussian components can hold on much more general open sets.
 Bogdan,  Kumagai and  Kwasnicki \cite{BKK2} established the non-scale invariant BHP  for a discontinuous
  Feller process $X$ having strong Feller property (called doubly Feller process)
    possibly having diffusive components  in weak duality with another  doubly Feller process
on open sets on metric measure  spaces, under  the condition that every semipolar set of $X$ is  polar,
 a comparability condition  of  the jump kernel and a Urysohn-type  cut-off function condition on the domains of the generator of
the  doubly Feller processes and their duals,  and a uniform off-diagonal boundedness on   Green functions;
see \cite[Assumptions A-D]{BKK2}.
 As a particular case,  the non-scale invariant BHP for $\Delta+\Delta^{\alpha/2}$ with $\alpha\in (0, 2)$ is shown to hold on any open set in $\R^d.$

 However, it is unknown whether the non-scale invariant BHP holds
   for more general discontinuous Hunt processes with diffusive part in open sets on metric measure spaces.
  The Urysohn-type  cut-off function condition  imposed in \cite[Assumption B]{BKK2}
    on the domains of the generator of the  doubly Feller processes and their duals
    played an essential role  in the approach there.
    This condition requires strong regularity on the generator of the  Feller  processes and their duals.
  For example, in the Euclidean case,   a large class of symmetric diffusion with jumps associated  with  generators
  having measurable coefficients  and non-symmetric diffusion with jumps with less regularity coefficients
  are excluded  from  \cite{BKK2}.
  The goal of this paper is to establish
    non-scale invariant BHP for  a more general class of discontinuous Hunt processes
  including discontinuous processes with diffusive components on metric measure spaces
 under a set of conditions that are strictly weaker than that of \cite{BKK2}.
In particular, we do not assume the Urysohn-type  cut-off function condition for the generator of the Hunt process
nor for its dual Hunt process.

\subsection{Main results}

In the following, we state the settings and our main results in detail.
 Let $(\sX,  d)$ be a locally compact separable metric space.
Let $X=(X_t, \zeta, \mathcal{F}_t, \P_x)$ be an  irreducible Hunt  process taking values in  $\sX$.
 The random variable $\zeta$ is the lifetime of the process $X$ so that $X_t=\partial$ for $t\geq \zeta,$ where $\partial$ is a cemetery  added to $\sX$ as a one-point compactification.
 Here we say  the Hunt process $X$ is  irreducible on $\sX$ if for any non-empty  open set $U\subset \sX,$
$$\P_x (T_U<\infty)>0 \quad \mbox {for every }  x\in \sX,$$
where $T_U:=\inf\{t>0: X_t\in U\}.$

 It is known that any Hunt process admits a L\'evy system that describes how the process jumps.
Suppose that the Hunt process $X$ has the L\'evy system $(J(x, dy), dt).$ That is, for each $x\in \sX, J(x, dy)$ is a measure on $\sX_\partial=\sX\cup \{\partial\}$ so that for any  nonnegative function $f$ on $\R_+\times
\sX_\partial\times \sX_\partial$ vanishing along the diagonal of $\sX_\partial \times \sX_\partial$, for any stopping time $T$
with respect to the minimal admissible  augmented  filtration generated by $X$ and $x\in \sX$,
$$
\E_x  \Big[ \sum_{s\leq T}f(s, X_{s-}, X_s); X_{s-}\neq X_s \Big]
 =  \E_x \Big[ \int_0^T\int_{ \sX_{\partial}} f(s,X_s,y) J(X_{s-}, dy)\,ds \Big].
$$

\begin{defn}\label{D:2.1} \rm
Suppose $U$ is an open subset of $\sX.$ A real-valued function $u$ defined on $\sX$ is said to be
\begin{itemize}
\item[\rm (i)]
 {\it  harmonic in   $U$ with respect to $X$}
if for every open set $B$  whose closure is a compact subset of $U,$
   $\E_x |u(X_{\tau_B})|<\infty$ and
 $u(x)=\E_x u(X_{\tau_B})$ for  each  $x\in B.$
  In particular, we say $u$  is   regular harmonic in   $U$ with respect to $X$
if    $\E_x |u(X_{\tau_U})|<\infty$ and
$u(x)=\E_x u(X_{\tau_U})$  for each  $x\in U.$

\item[\rm (ii)]
 {\it $\alpha$-excessive with respect to $X^U$ with $\alpha\geq 0$}
if $u$ is non-negative,
$ u(x)\geq e^{-\alpha t}\E_x u(X^U_t)$ and $u(x)=\lim_{t\downarrow 0} e^{-\alpha t}\E_x  \left[ u(X^U_t)\right]$
 for every $t>0$ and   $x\in U.$
 If $\alpha=0,$ we say $u$ is an excessive function with respect to $X^U.$

\end{itemize}
\end{defn}

\smallskip

  We refer the reader to \cite{C} for the equivalence of the analytic and probabilistic notions of the harmonicity for symmetric Hunt processes on
 metric measure spaces.  For an open subset $U$ of $\sX,$ we use $\overline U$ to denote its closure.
 Let $U$ be an open subset of $\sX.$
 Denote by $X^U$ the subprocess of $X$ killed upon exiting $U.$
 The  process $X^U$ is a  Hunt process on $U$; see, for example,  \cite[Exercise 3.3.7]{CF}.
 We   consider the following assumptions on the  Hunt process $X$ in this paper.

 \begin{enumerate}

\item[{\bf (A1)}]
 Hunt's hypothesis holds;  that is, every semi-polar set of $X$ is polar.
 Suppose further that for each ball $B$  so that  $\overline B\subsetneq \sX$,
 there exists a Radon measure $m_B$ on $B$ such that  the process $X^B$ is   in weak  duality
  to another
   Hunt process $\wh X^B$ with respect to  the Radon measure $m_B$.
   That is, for each  nonnegative Borel  functions $f, g$ in $\sX$ and for each $\alpha>0,$
\begin{equation}\label{e:1.1'}
\int_B  G^{\alpha, B} f(x)\cdot g(x) m_B(dx)=\int_B  f(x) \wh G^{\alpha, B} g(x) m_B(dx),
\end{equation}
where $G^{\alpha, B} f(x):=\E_x \int_0^\infty e^{-\alpha t} f(X^B_t)\,dt$ and $\wh G^{\alpha, B} f(x):=\E_x \int_0^\infty e^{-\alpha t} f(\wh X^B_t)\,dt.$  Moreover, for each $x\in B$ and $\alpha>0,$  $G^{\alpha, B}(x, \cdot)$ and $\wh G^{\alpha, B}(x, \cdot)$ are absolutely continuous with respect to $m_B(dy).$

\end{enumerate}

 Throughout this paper, we always use $m_B$ to denote the reference measure in assumption (A1).
 Note that in this paper, we only assume for each ball $B$ so that  $\overline B\subsetneq \sX$,
$X^B$ has a weak dual; we do not assume $X$ itself has a weak dual.
We mention that if
 there exists a Radon measure $m$ on $\sX$ such that
 the process $X$ is   in weak  duality
  to another Hunt process $\wh X$ with respect to   $m,$ then by \cite[Theorem 13.25 and Remark 13.26]{ChungW}, for each  open subset $U$ with
  $\overline U\subsetneq \sX,$
   the semigroups  of $X^U$ and  $\wh X^U$ are in weak duality relative to $m$ on $U.$
  Hence \eqref{e:1.1'} holds.

Let $B$ be a ball  so that $\overline B\subsetneq \sX$, and $U$ an open subset of $B$.
Denote by $\wh X^{B, U}$ the subprocess of $\wh X^B$  killed upon leaving $U.$
Observe that $X^U$
 can be viewed as the subprocess of $X^B$ killed upon leaving $U.$
 By \cite[Theorem 13.25 and Remark 13.26]{ChungW},
   the semigroups  of $X^U$ and  $\wh X^{B, U}$ are in weak duality relative to $m_B$ on $U.$
   That is, for each  nonnegative Borel  functions $f, g$ in $U$ and for each $\alpha>0,$
$$
\int_U  G^{\alpha, B}_U f(x)\cdot g(x) m_B(dx)=\int_U  f(x) \wh G^{\alpha, B}_U g(x) m_B(dx),
$$
where $G^{\alpha, B}_U f(x)  := G^{\alpha, U} f(x)
:=\E_x \int_0^\infty e^{-\alpha t} f(X^U_t)\,dt$ and $\wh G^{\alpha, B}_U f(x):=\E_x \int_0^\infty e^{-\alpha t} f(\wh X^{B, U}_t)\,dt.$
Note that for each $x\in U$ and $\alpha>0,$  $G^{\alpha, B}_U(x, \cdot)$ and $\wh G^{\alpha, B}_U(x, \cdot)$ are absolutely continuous with respect to $m_B(dy)$  restricted on $U.$
  For   each $\alpha>0$,
 there exists a unique $\alpha$-potential kernel $G^{\alpha, B}_U(x, y)$ defined on $U\times U$ such that for  each nonnegative Borel  function $f$ on $U,$
\begin{equation}\label{e:G1}
G^{\alpha, B}_U f(x)=\int_U G^{\alpha, B}_U(x, y) f(y) m_B(dy), \quad \wh G^{\alpha, B}_U f(x)=\int_U G^{\alpha, B}_U(y, x) f(y) m_B(dy),
\end{equation}
and  $x\mapsto G^{\alpha, B}_U(x, y)$ is $\alpha$-excessive with respect to $X^U,$  $y\mapsto G^{\alpha, B}_U(x, y)$ is $\alpha$-excessive with respect to $\wh X^{B, U};$ see \cite[Theorem 13.2]{ChungW},   or \cite[Theorem 1]{KW},  or the remarks after \cite[Proposition VI.1.3]{BG}.
Since $G^{\alpha, B}_U(x, y)$ is non-increasing in $\alpha>0$ for $(x, y)\in U\times U,$ we define the {\it potential kernel}
$G^B_U(x, y):=\lim_{\alpha\rightarrow 0}G^{\alpha, B}_U(x, y)$ for $(x, y)\in U\times U,$ which exists  uniquely.
 We set $G^B_U(x, y)=0$ for $(x, y)\notin U\times U.$
By monotone convergence theorem, for each nonnegative Borel function $f$ on $U,$
$$G^B_U f(x)=\int_U G^B_U(x, y) f(y) m_B(dy)=\E_x \int_0^{\tau_U} f(X_t) dt.$$
Since the increasing limit of excessive function is also an excessive function, for each  $y\in U,$ $x\mapsto G^B_U(x, y)$ is excessive with respect to $X^U,$ and for each  $x\in U,$  $y\mapsto G^B_U(x, y)$ is excessive with respect to $\wh X^{B, U}$.

\begin{defn}\label{De:1}
Let $B$ be a ball in $\sX$ with  $\overline B\subsetneq \sX.$
 For an open subset $U\subset B,$
  we say $G^B_U(x, y)$ is
the Green function of $X^B$ in $U$ (or equivalently $X^U$)  with respect to the measure $m_B$ if it is  not identically infinite and is  the potential kernel of $X^B$ in $U$. That is, the following conditions hold:
\begin{enumerate}
 \item[\rm (i)]
 For any Borel bounded function  $f\geq 0$ on $U$,
$$
\E_x \int_0^{\tau_U} f(X_s) ds=\E_x \int_0^{\tau_U} f(X^B_s) ds = \int_U G^B_U (x, y) f(y) m_B(dy), \quad \hbox{for} \quad x\in U.
$$

\item[\rm (ii)]
For each fixed $y\in U,$ $x\mapsto G^B_U(x, y)$ is excessive with respect to $X^B$ in $U.$ For each fixed $x\in U,$  $y\mapsto G^B_U(x, y)$ is excessive with respect to $\wh X^{B, U}.$

 \end{enumerate}

  When   $U=B$,  we denote the Green function $G^B_U(x, y)$ by $G^B(x, y)$.

\end{defn}

Let $U$ be an open subset of $\sX$. For each $x\in \sX$ and $R>0,$ define $B_U(x, R):=U\cap B(x, R).$
For each ball $B=B(x_0, R_0)$  centered at $x_0$  with radius $R_0$
  and $\lambda >0,$  we use $\lambda B$ to denote the concentric ball $B(x_0, \lambda R_0).$
  For each $x_0\in \sX$ and $0<r<R<\infty,$  let
  $$
  A(x_0, r, R):=B(x_0, R) \setminus \overline{B(x_0, r)}.
  $$
  For each ball $B\subset \sX$ and  each open subset $U$ of $B,$ denote by $\tau_U$ and $\wh \tau^B_U$ the first exiting times of $X$ (equivalently, of $X^B$) and the dual process $\wh X^B$ from $U$ respectively.

\begin{defn}
\begin{enumerate}
 \item[\rm (i)]
 We say condition   ${\bf (ED)} _{\leq, x_0, R}$ holds
  if for the  ball $B=B(x_0, R)$ with $\overline B \subsetneq \sX$,
  there is a constant $C_1=C_1(x_0, R)>0$  so that   for
      any  open subset $ U\subset 2^{-1}B$ and every $  x\in B_U(x_0,  R/64)$,
\begin{equation}\label{e:ED1}
\P_x  \big( X^{B}_{\tau_{B_U(x_0, R/32)}} \in U \big)
\leq   C_1\E_x[\tau_{B_U(x_0,  R/32)}]  .
 \end{equation}

  \item[\rm (ii)]
We say condition  ${\bf {(\wh{ED})}} _{\leq, x_0, R}$ holds
 if for the  ball $B=B(x_0, R)$ with $\overline B \subsetneq \sX$,
there is a constant $C_2=C_2(x_0, R)>0$  so that  for
 any  open subset $  U\subset 2^{-1}B$ and
every $x\in  U \cap A(x_0, 3R/8, R/2)$,
\begin{equation}\label{e:ED2}
\P_x  \Big( \wh X^{B}_{\wh\tau^B_{U\cap  A(x_0, R/4, R/2)}}  \in U \Big)
\leq   C_2\E_x \left[ \wh\tau^{B}_{U\cap  A(x_0, R/4, R/2)} \right] .
 \end{equation}

\item[\rm (iii)] We say condition  ${\bf (ED)}^s _{\leq, R}$
  (resp. ${\bf (\wh {ED} ) }^s_{\leq, R}$) holds
 if  \eqref{e:ED1} holds  with $C_1=C_1 (R)$ (resp.  \eqref{e:ED2} holds with constant $C_2=C_2(R)$)
for any ball $B$ with radius $R$  and  $\overline B \subsetneq \sX$.

 \end{enumerate}

\end{defn}

\begin{remark}\label{R:1.1}\rm

In the definition of ${\bf (ED)} _{\leq, x_0, R}$, as the open subset $U\subset 2^{-1}B,$
 one can write $X$ in place of $X^B$ in \eqref{e:ED1}. That is,
   \eqref{e:ED1} is equivalent to that
\begin{equation}\label{e:1.5}
\P_x  \big( X_{\tau_{B_U(x_0, R/32)}} \in U \big)
\leq   C_1\E_x[\tau_{B_U(x_0, R/32)}] \quad\hbox{ for every }  x\in B_U(x_0, R/64).
 \end{equation}
If there exists a Radon measure $m$ on $\sX$ such that
 the process $X$ is   in weak  duality
  to another Hunt process $\wh X$ with respect to   $m,$ then
  a  similar remark applies to ${\bf {(\wh{ED})}} _{\leq, x_0, R}$ as well.

\end{remark}

We consider the following assumptions (A2)-(A5). Let $R_0>0.$

\begin{enumerate}

\item[{\bf (A2)}]
 For each  $R\in (0, R_0)$ and $x_0\in\sX$ with $\overline{B(x_0, R)}\subsetneq \sX,$
 ${\bf (ED)} _{\leq, x_0, R}$
  and ${\bf {(\wh{ED})}} _{\leq, x_0, R}$ hold.

\item[{\bf (A3)}]
For each ball  $B=B(x_0, R)$ centered at $x_0\in\sX$ with radius $R\in (0, R_0)$ so that $ \overline{B}\subsetneq \sX,$
\begin{equation}\label{e:G}
\sup_{B(x_0, R/16) \times  A(x_0, R/8, R/2) }
 G^{B}(x, y)  <\infty.
\end{equation}

\item[{\bf   (A4)}]
 For each ball  $B=B(x_0, R)$ centered at $x_0\in\sX$ with radius $R\in (0, R_0)$ so that $ \overline{B}\subsetneq \sX$,
  there exists $ c_1=c_1(x_0, R)>1$  such that  for any $x\in   2^{-1}B, $
\begin{equation}\label{e:J1}
c_1^{-1}J(x_0, dy)\leq J(x, dy) \leq c_1J(x_0, dy)   \quad \hbox{on }  B^c.
\end{equation}
Furthermore,
 there is a strictly positive function  $(x, y)\mapsto J^{B}(x, y)$ in $B \times B$
 such that  $J^{B}(x, y)$ is finite for $x, y\in B$ with $x\neq y$ and
 $$
  J(x, dy)=J^{B}(x, y)\,m_{B}(dy) \quad  \mbox{on} \quad B\times B.
  $$
 Moreover,
  \begin{equation}\label{e:J3}
  0<\inf_{y\in   B(x_0, R/2)\setminus B(x_0, R/16)} J^{B}(x_0, y)\leq \sup_{y\in  B(x_0, R/2)\setminus B(x_0, R/16)  } J^{B}(x_0, y)<\infty.
  \end{equation}
   For each $ 0<p<q<1,$
  there exists $c_2=c_2(x_0, p, q, R)>1$  such that  for any $x\in   pB$ and $y\in B\setminus (qB),$
\begin{equation}\label{e:J2}
c_2^{-1}J^{B}(x_0, y)\leq J^{B}(x, y) \leq c_2J^{B}(x_0, y).
\end{equation}

\item[{\bf (A5)}]
For each  $R\in (0, R_0),$   there exists $C=C(R)>0$ such that
$$
   \sup_{x\in \sX}\E_{x} \tau_{B(x, R)}\leq C.
$$

\end{enumerate}

We say  assumptions (A2'), (A3') and (A4') hold for  $R_0>0,$ if
\begin{enumerate}

\item[{\bf (A2')}] For each $R\in (0, R_0),$  ${\bf (ED)}^s _{\leq, R}$
and ${\bf (\wh {ED} ) }^s_{\leq, R}$ hold.

\item[{\bf (A3')}]  For each $R\in (0, R_0),$ the condition \eqref{e:G} in assumption (A3)  holds uniformly for $x_0\in \sX.$

\item[{\bf(A4')}]  Assumption (A4) holds and the conditions \eqref{e:J1}-\eqref{e:J2} hold uniformly for $x_0\in \sX.$

\end{enumerate}

  Assumptions (A2')-(A4') are slightly stronger than (A2)-(A4) respectively.

\begin{remark}\label{R:1.2}\rm

\begin{enumerate}

\item[(i)]
 If  $X$ is a  doubly Feller process
in weak duality with another doubly Feller process $\wh X$
and  the  Urysohn-type  cut-off function condition holds for  both $X$ and $\wh X$,
  one can choose bump functions $f$ and $g$ such that $f=1$ on $\overline {A(x_0, R/32, R/2)}$ and $f=0$ on $\sX\setminus A(x_0, R/64, R);$ and
$g=1$ on $\overline {B(x_0, R/4)}$ and $g=0$ on $\sX\setminus B(x_0, 3R/8)$.
 Then  by an  argument similar to that of Lemma 3.1 in \cite{BKK2}, \eqref{e:ED1} and \eqref{e:ED2} hold.
     Hence Assumption (A2)    is weaker than Assumption B   on  Urysohn-type  cut-off function condition for  both $X$ and $\wh X$ of
     \cite{BKK2}.
 In Proposition \ref{P1} below, we present  some sufficient conditions for (A2) that are easier to verify in concrete cases.

\item[(ii)]  Note that in Assumptions (A3)  and (A3'), we do not assume the  corresponding
condition on the Green function of the dual process  $\wh X^{B}$.
  Hence they are weaker than
   Assumption D imposed on the Green functions of
   $X^{B}$ and $\wh X^{B}$
   in \cite{BKK2}.  Assumptions   (A3) and  (A5) are mild conditions that are satisfied by many Hunt processes.

\item[(iii)]
 In Assumption (A4), we do not assume the comparability of the jump density function of the dual process
  $\wh X^B$.
The condition  \eqref{e:J3}
  in assumption (A4)
   is a mild condition  that is satisfied by many Hunt processes.
In particular, if Assumption C  of \cite{BKK2}  on  the comparability of the  jump density kernel of the dual process holds, then  the infimum part of \eqref{e:J3} holds by  \cite[(2.8)]{BKK2}  and its supremum part  holds by a similar covering argument as that of  \cite[(2.8)]{BKK2}.
Thus  Assumption  (A4) is weaker than Assumption C in \cite{BKK2}.

\item[(iv)]  In summary, our conditions (A1)-(A4), under which we establish BHP in Theorem \ref{T0} of
this paper,  are strictly weaker than Assumptions A-D of \cite{BKK2}.

\end{enumerate}
\end{remark}

  \medskip

For a Borel set $A\subset \sX$,  define $T_A:=\inf\{t>0: X_t\in A\}.$

 \begin{defn}\label{D:1} \rm
A point $y$ is said to be regular of $X$ for a Borel set $A$ if $\P^y(T_A=0)=1.$
 Denote by $A^r$ all the regular points of $A$ with respect to $X$.
\end{defn}

\smallskip

The following is the main result of this paper.

\begin{thm}\label{T0}
Suppose that assumptions (A1) and (A2)-(A4)  hold for some $R_0>0$.
Let $D$ be an  open set in $\sX$.
Then for each  $z_0\in \partial D$ and $R\in (0,  R_0/16 ) $  so that $\overline{B(z_0, 2R)}\subsetneq \sX$,
 there exists a constant
$C=C( z_0, R)>0$,
 which depends on the constants in assumptions (A2)-(A4)  for the ball $B(z_0,  2R)$ and is independent of the geometry of $D$,
   such that   for any  nonnegative
 regular harmonic functions $f$ and $g$
on $D\cap B(z_0, R)$   vanishing  on  $  (D^c)^r \cap B(z_0, R),$
\begin{equation}\label{e:1}
f(x)g(y)\leq Cf(y)g(x)  \quad \hbox{for } x, y\in D\cap B(z_0, R/32).
\end{equation}
 Furthermore, assume  that the assumptions (A1), (A2'), (A3'), (A4') and (A5) hold for  $R_0>0$.
 Then the constant in \eqref{e:1} depends only on $R,$ that is,  for each $R\in (0, R_0/16)$, the BHP \eqref{e:1} holds
  with $C=C(R)\geq 1$ uniformly for $z_0\in \partial D.$
\end{thm}

 \begin{remark}\rm

Hunt's hypothesis in assumption (A1) is  not needed in Theorem \ref{T0} if we consider the  BHP  \eqref{e:1} only  for those
nonnegative  regular harmonic functions  $f$ and $g$
that vanish  on $D^c\cap B(z_0, R)$.
The same remark applies to other results in this paper that involves the regular points of $D^c$.

\end{remark}

 We give some sufficient conditions   for Assumption (A2) that are easy  to verify.
 The following conditions are motivated by Proposition 3.1 in \cite{CC}.

\begin{defn} \begin{enumerate}
\item [\rm (i)]  We say condition  ${\bf (EP)} _{\leq, r_0, \gamma}$ holds
with  $r_0>0$ and $\gamma>0$  if there is a constant $C=C(r_0, \gamma)>0$   such that for every  $x\in \sX, $ $t>0$ and $ r\in (0, r_0], $
\begin{equation}\label{e:ept0}
\P_x\big(\tau_{B(x,r)}<t\big)\leq  Ct/r^{\gamma}.
\end{equation}

  \item [\rm (ii)]
 We say condition  ${\bf ({\wh{EP}})} _{\leq,  r_0, \gamma}$ holds
   with $r_0>0$  and $\gamma>0$  if there is a constant $C=C(r_0, \gamma)>0$   so that for
 every ball $B$ with radius $r_0$  and  $\overline B\subsetneq \sX$,
  the following holds for every  $x\in 2^{-1}B$, $r\in (0, r_0/4)$ and $t>0$,
  \begin{equation}\label{e:ep1}
\P_x\big(\wh\tau^B_{B(x,r)}< t\wedge \wh\tau^B_{2^{-1}B}  \big)\leq  Ct/r^{\gamma}  .
\end{equation}
 \end{enumerate}
\end{defn}

\medskip

Note that the Hunt process $X$ has a L\'evy system $(J(x, dy), dt).$ It is easy to see that for each  ball $B$ in $\sX,$ the subprocess $X^B$ has a L\'evy system $(J^B(x, dy), dt),$ where
$J^B(x, dy)=J(x, dy)$ on $B\times B$ and
$J^B(x, \{\partial\} )=J(x, \sX_\partial \setminus B)$.
 It follows from \cite{G} that a L\'evy system  $(\wh J^B(x, dy), \wh H^B)$ holds for $\wh X^B$ that satisfies $\wh H^B_t=t$ and
$$
\wh J^B(y, dx) m_B(dy)= J^B(x, dy) m_B(dx)=J^B(x, y) m_B(dx) m_B(dy) \quad \hbox{on } B\times B.
$$
 Hence we can take $\wh J^B(y, dx)=J^B(x, y) m_B(dx)$  on $B\times B.$

\begin{defn}  \begin{enumerate}
\item [\rm (i)]
 We say condition  ${\rm \bf (Jt)}_{\leq, \bar r, \gamma}$ holds
 with some $\bar r>0$ and $\gamma>0$ if  there is a constant $C=C(\bar r, \gamma)>0$ such that for $r\in (0, \bar r],$
\begin{equation}\label{e:Jt}
J(x,  \sX \setminus B(x,r))\leq Cr^{-\gamma}   \quad\hbox{ for every }x\in\sX.
\end{equation}

\item [\rm (ii)]
 We say  condition  ${\rm \bf ({\wh{Jt}})}_{\leq,  \bar r,  \gamma}$ holds
   with $ \bar r>0$ and $\gamma >0$ if there is a constant $C= C(\bar r, \gamma)>0$ such that for every ball $B$ with radius
 $\bar r$ and  $\overline B\subsetneq \sX$ and   every  $r\in (0, \bar r/4],$
\begin{equation}\label{e:Jt1}
\wh J^B(x, (2^{-1}B)\setminus B(x,r))\leq Cr^{-\gamma}   \quad\hbox{ for every }x\in 2^{-1}B.
\end{equation}
\end{enumerate}
\end{defn}

\medskip

\begin{remark}\label{R:1.11} \rm
 Clearly,  conditions ${\bf ({\wh{EP}})} _{\leq,  r_0, \gamma}$  and ${\rm \bf ({\wh{Jt}})}_{\leq,  \bar r,  \gamma}$
 are weaker than their dual counterparts  ${\bf (EP)} _{\leq, r_0, \gamma}$ and ${\rm \bf (Jt)}_{\leq, \bar r, \gamma}$.
In particular, when $X$ is a $m$-symmetric Hunt process, property
 ${\bf (EP)} _{\leq, r_0, \gamma}$ implies  ${\bf ({\wh{EP}})} _{\leq,  r_0, \gamma}$
 and property ${\rm \bf (Jt)}_{\leq, \bar r, \gamma}$ yields ${\rm \bf ({\wh{Jt}})}_{\leq,  \bar r,  \gamma}$.
\end{remark}

\medskip

By  Proposition 3.1  of Cao and Chen \cite{CC},
if  the  conditions ${\rm \bf (Jt)}_{\leq, R_0, \gamma}$  and  ${\bf (EP)} _{\leq, R_0, \gamma}$
hold for some
 $R_0>0$ and $\gamma>0,$  then  \eqref{e:1.5} holds
 with $C_1=C_1(R_0, \gamma)$ independent of $x_0\in \sX$ and $R\in (0, R_0]$  for any ball $B=B(x_0, R)$
 and so ${\bf (ED)}^s _{\leq, R}$   holds for
 $X$   with constant $C$  uniform in $R\in (0, R_0]$.
By   an argument similar to that of \cite[Proposition 3.1]{CC}, we can show that
the  conditions ${\rm \bf (\wh{Jt})}_{\leq, R,  \gamma}$  and  ${\bf (\wh{EP})} _{\leq, R, \gamma}$
 imply  that ${\bf (\wh{ED})}^s _{\leq, R}$   holds.

\medskip

 The following result gives some sufficient conditions for (A2').

\begin{prp}\label{P1}
Let $R_0>0.$ Suppose that
  for each $R\in (0, R_0)$, there is some $\gamma >0$ so that
  {\rm ${\bf (EP)} _{\leq , R, \gamma}$},
${\bf (Jt)}_{\leq, R, \gamma}$,
 {\rm ${\bf (\wh{EP})} _{\leq, R, \gamma}$} and
${\bf (\wh{Jt})}_{\leq,  R, \gamma}$ hold.
 Then assumption (A2') holds.
\end{prp}

\smallskip

 By Proposition \ref{P1} and Theorem \ref{T0},  we have the following.

\begin{cor}\label{C1}
 In the statement of Theorem \ref{T0}, if we replace assumptions (A2) and (A2')
  by the condition  that for each $R\in (0, R_0)$, there is some $\gamma >0$ so that
  {\rm ${\bf (EP)} _{\leq , R, \gamma}$},
${\bf (Jt)}_{\leq, R, \gamma}$,
 {\rm ${\bf (\wh{EP})} _{\leq, R, \gamma}$} and
${\bf (\wh{Jt})}_{\leq,  R, \gamma}$  hold,
 then the conclusion of Theorem \ref{T0} holds.
\end{cor}

\begin{remark}\label{r1}\rm

\begin{enumerate}
  \item[(i)]
  In Theorem \ref{T0}, we assume assumption (A2)  in place of the  Urysohn-type  cut-off function condition
  on the domains of the generators of the processes and their duals in \cite{BKK2}.
  As we mentioned in Remark \ref{R:1.2}, assumption (A2)
  is weaker than the  Urysohn-type  cut-off function condition  on the generators of
  $X$  and its dual process in \cite[Assumption B]{BKK2}.
  In  Propositions \ref{C2} and \ref{C6}  of Section 4, we show by using Proposition \ref{P1}  that  the BHP \eqref{e:1}  in particular  holds for   any symmetric diffusion with jumps associated with a regular symmetric Dirichlet form on
  a metric measure space having both the strongly local term and  the pure-jump term that admits a two-sided  heat kernel estimates of the mixture of (sub-)Gaussian and stable-like form, as well as
   for a wide class of non-symmetric diffusion processes with jumps in non-divergence form in $\R^d$.
    Those processes are not covered  in \cite{BKK2}.

\item [(ii)]
   Theorem \ref{T0} further shows that  if  assumptions (A1), (A2')-(A4') and (A5) hold, then the multiplicative constant
   $C$ of the BHP in Theorem \ref{T0}
    holds uniformly in $z_0\in \partial D$  and the geometry of open set $D$ and
   depends only on the scale constant $R.$
   This may be useful in the study of Dirichlet heat kernel and Green function estimates for such diffusion with jumps in open sets
    on  metric measure spaces.  We remark that
by carefully tracking the constants of the BHP  in \cite{BKK2} (see, e.g. Theorem 3.5 in \cite{BKK2}),
 the result  implicitly  implies that   if the assumptions in \cite{BKK2} hold uniformly for $x\in\sX$,
then  for a large class of L\'evy processes in $\R^d,$  the multiplicative constant of the BHP  there
  only depends on $R.$
In this paper, we  establish such BHP for a large class of   discontinuous processes on
 metric measure spaces in Theorem \ref{T0},
  which may not satisfy the Urysohn-type  cut-off function condition.
In particular, we show that  such  BHP  holds for a large class of symmetric diffusion with jumps  associated  with  generators
  having measurable coefficients
 on $\R^d$,  for
symmetric diffusions with jumps on $d$-set with walk dimension $\beta \geq 2 $,  and  for
a class of non-symmetric diffusion with jumps in $\R^d$ (see Examples 4.2 and 4.3 and   Proposition \ref{C6}).

\item [(iii)]
 The positivity of jump density function of $X^B$ in assumption (A4) is crucial for Theorem \ref{T0};
  see, e.g., the proofs for Lemmas \ref{L:2.1n} and \ref{L:3.2}.

 \item[(iv)]
We show in subsection 4.3 that  the BHP  holds for a  L\'evy process with degenerate Gaussian component in open sets in $\R^d$ associated with the generator
$\sL:=\sum_{i=1}^m\partial_{x_i}^2+\Delta^{\alpha/2},$  where $d>m\geq 3$ and $\Delta^{\alpha/2}  := - (-\Delta)^{\alpha/2}$
is the fractional operator in $\R^d.$
Observe that the differential operator component of $\sL$ is   degenerate on $\R^d$ as $m<d$.
This example seems to be new and is of independent interest.

  \end{enumerate}
  \end{remark}

\subsection{Idea of  our approach}

 The key ingredient  of the approach in \cite{BKK2}  is a local supremum estimate for sub-harmonic functions, see Theorem 3.4 in \cite{BKK2}, which is obtained by a  regularization procedure of harmonic measure. The assumption of the Feller and strong Feller property of the process and the Urysohn-type  cut-off function condition on the domains of the generator of the processes and their dual played a crucial role in the approach of \cite{BKK2}.

 In this paper,
we mainly employ the representation formula of regular harmonic functions for discontinuous Hunt processes in open sets vanishing in part of the boundary of open sets in terms of  Green functions
  established in our previous paper \cite{CW2} (see also Proposition \ref{P:2.2} below).
  Our main idea is to prove that the ratio of the Green functions
  for $X$ in an open subset with  fixed scale $R$  near the boundary
  is comparable to the ratio of the  expected exit times.
   To achieve this, our approach  is to obtain the two-sided estimates of the Green function $G^B_{D\cap B(z_0, R/2)}(x, y)$
    for $x\in D\cap B(z_0, R/64)$ and $y\in D\cap A(z_0, 3R/8, R/2)$ under assumption (A2) on the exit distribution of $X^B$ and $\wh X^B$ from open sets,
    where $z_0\in \partial D$ and $B:=B(z_0, R)$  (see Lemma \ref{L:3.2} and Proposition \ref{P:2.3}).
  By this comparability and the representation formula of regular harmonic functions
  in Proposition \ref{P:2.2},
   the BHP in Theorem \ref{T0} can be established.
Furthermore we prove that the decay rate of any positive harmonic function for $X$ is in fact comparable to that of the
 expected exit times.

In this paper, we use $C_c(\R^d)$ and $C_b(\R^d)$
 to denote the space of continuous functions on $\R^d$ with compact support and the space of bounded continuous functions on $\R^d$, respectively.
The space of functions in $C_c(\R^d)$ that have continuous
first derivatives (resp. continuous derivatives of any order) is denoted by $C_c^1 (\R^d)$  (resp. $C_c^\infty (\R^d))$.
The space of functions in $C_b(\R^d)$ that have bounded continuous derivatives up to and including second order
is denoted by $C^2_b(\R^d)$.
  For  two functions $f$ and $g$,  notation  $f\asymp g$ means that there exists $c>1$ so that $c^{-1}g\leq f\leq cg$
  on their common domain of definitions.

The rest of this paper is organized as follows. Section \ref{S:2} is devoted to the proof of Theorem \ref{T0}. In Section \ref{S:3}, we  prove  Proposition \ref{P1}. In Section \ref{S:4},
we    present some examples to illustrate the strength of  Theorem \ref{T0}.

 \section{ Boundary Harnack principle}\label{S:2}

 In this section, we present the proof for  Theorem \ref{T0}.
 Let $(\sX,  d)$ be a locally compact separable metric space,
and
 $X=(X_t, \zeta, \mathcal{F}_t, \P_x)$ be an irreducible  Hunt  process taking values in  $\sX.$ Suppose that the Hunt process $X$ has the L\'evy system $(J(x, dy), dt).$

\begin{lem}\label{L:2.1n}
Suppose assumption (A4) holds for some $R_0>0$. For each  $x_0\in \sX$ and $R\in (0, R_0/8),$   there exists $C=C(x_0, R)>0$ such that
\begin{equation}\label{e:2.1}
\E_x \left[ \tau_{B(x_0, R)} \right]\leq C  \quad \mbox{for} \quad x\in B(x_0, R/2).
\end{equation}
\end{lem}

\pf Let $x_0\in \sX$ and $R\in (0, R_0/8).$   By the L\'evy system formula of $X$ and  the condition  \eqref{e:J1} of assumption (A4), there exists $c_1=c_1(x_0, R)>0$ such that for
$ x\in B(x_0, R/2),$
\begin{equation}\label{e:2.1a}\begin{aligned}
1&\geq \P_x(X_{\tau_{B(x_0, R)}}\in  B(x_0, 4R)\setminus B(x_0, 2R))\\
&=\E_x \int_0^{\tau_{B(x_0, R)}}\int_{B(x_0, 4R)\setminus B(x_0, 2R)} J(X_s, dy)\,ds\\
&\geq c_1\E_x \int_0^{\tau_{B(x_0, R)}}\int_{B(x_0, 4R)\setminus B(x_0, 2R)} J(x_0, dy)\,ds\\
&= c_1\E_x \left[ \tau_{B(x_0, R)} \right] J(x_0, B(x_0, 4R)\setminus B(x_0, 2R)).
\end{aligned}
\end{equation}
Note that  by  assumption (A4) for $R_0$, there is a strictly positive jump density function $J^{B(x_0, 4R)}(x, y)$ such that  $J(x, dy)=J^{B(x_0, 4R)}(x, y) m_{B(x_0, 4R)}(dy)$ on $B(x_0, 4R)\times B(x_0, 4R).$
Hence $c_2(x_0, R):=J(x_0, B(x_0, 4R)\setminus B(x_0, 2R))>0.$ This together with \eqref{e:2.1a} yields that \eqref{e:2.1} holds.
\qed

\begin{lem}\label{L:3.2}
Suppose assumptions (A1) and (A2)-(A4) hold for some $R_0>0$.
 Let $D$ be an  open subset of $\sX$. For each
 $B:=B(z_0, R)$ with $z_0\in \partial D$ and
 $R\in (0, R_0 /8)$
 so that  $\overline {B}\subsetneq \sX,$
there exists $C=C(z_0, R)>0$ depending on the constants in (A2)-(A4) for the ball $B(z_0, R)$
such that for any
$x\in D\cap B(z_0,  R/64)$ and $y\in D\cap  A(z_0, 3R/8, R/2), $
\begin{equation}\label{e:2.2n}
\begin{aligned}
&C^{-1}\E_x\tau_{D\cap B(z_0, R/32)}\int_{D\cap (B(z_0, R/2)\setminus B(z_0, R/16))}  G^B_{D\cap B(z_0, R/2)} (u, y)\, m_B(du)
\leq
 G^B_{D\cap B(z_0, R/2)}(x, y)\\
 &\qquad \leq C\E_x\tau_{D\cap B(z_0, R/32)}\int_{D\cap (B(z_0, R/2)\setminus B(z_0, R/16))}  G^B_{D\cap B(z_0, R/2)}(u, y)\, m_B(du),
\end{aligned}
\end{equation}
where
$G^B_{D\cap B(z_0, R/2)}(x, y)$ is the Green function of $X^B$ in $D\cap B(z_0, R/2)$ with respect to the reference measure $m_B(dy).$
 If we further assume   (A2')-(A4') and (A5) hold, then the constant in \eqref{e:2.2n} depends only on $R.$
\end{lem}

\pf  Let $z_0\in \partial D$ and $R\in (0,  R_0/8).$
Let $B:=B(z_0, R).$
For the simplicity of notation, for each $s>0,$ let $D_s:=D\cap B(z_0, s).$
 Note that if $D\cap A(z_0, 3R/8, R/2)=\varnothing,$ then
  $ G^B_{D_{R/2}}(u, y)=0$ for any
$(u, y)\in D_{R/2}\times  (D\cap  A(z_0, 3R/8, R/2)).$
 Thus  each of the three terms in  \eqref{e:2.2n} are   zero and so \eqref{e:2.2n} holds.
In the following, we always assume $D\cap A(z_0, 3R/8, R/2)\neq \varnothing.$

By assumption (A1), the process $X^B$ is  in weak duality to another
   Hunt process $\wh X^B$ with respect to  the Radon measure $m_B$.
   By the argument below assumption (A1),
   for any open subset $U\subset B,$   the Green function  $\wh G^B_U(x, y)$ of  $\wh X^B$ in $U$  exists uniquely and $\wh G^B_U(x, y)=G^B_U(y, x)$ for $x, y\in U\setminus {\rm diag}.$
Recall that  there is a L\'evy system $(\wh J^B(x, dy), \wh H^B)$ for $\wh X^B$ satisfying $\wh H^B_t=t$ and
$\wh J^B(x, dy) = J^B(y, x) m_B(dy) $ on $B.$

\smallskip

Let $x\in D_{R/64}$ and $y\in D\cap  A(z_0, 3R/8, R/2).$
 By  the proof for
  (2.18) in \cite{CW2} and the argument below it,
 the Green function $G^B_{D_{R/2}}(\cdot, y)$ satisfies that
\begin{eqnarray}\label{e:2.4}
&&  G^B_{D_{R/2}}(x, y) =\E_x [G^B_{D_{R/2}}(X_{\tau_{D_{R/32}}}, y)] \nonumber \\
&=& \E_x [G^B_{D_{R/2}}(X_{\tau_{D_{R/32}}}, y); X_{\tau_{D_{R/32}}}\in  D_{R/16}] \nonumber \\
&&+\E_x [G^B_{D_{R/2}}(X_{\tau_{D_{R/32}}}, y); X_{\tau_{D_{R/32}}}\in D_{R/2}\setminus D_{R/16}] \nonumber \\
&:=& I_1(x, y)+I_2(x, y).
\end{eqnarray}
For the first item, by assumption (A2),   ${\bf (ED)} _{\leq, z_0, R}$   holds and  thus there exists $c_1=c_1(z_0, R)>0$ such that
\begin{equation}\label{e:2.2}\begin{aligned}
I_1(x, y)&\leq \sup_{u\in D_{R/16}} G^B_{D_{R/2}}(u, y)\cdot \P_x(X_{\tau_{D_{R/32}}}\in D_{R/16})\\
&\leq c_1 \E_x \left[ \tau_{D_{R/32}} \right]  \sup_{u\in D_{R/16}} G^B_{D_{R/2}}(u, y).
\end{aligned}\end{equation}
For each open set $U\subset B,$ denote by $\wh\tau^B_U$ the first time of $\hat X^B$ exiting from $U.$ In the following, we prove that there exists $c_2=c_2(z_0, R)>0$ such that for $y\in D\cap A(z_0, 3R/8, R/2),$
\begin{equation}\label{e:2.3}
\sup_{u\in D_{R/16}}G^B_{D_{R/2}}(u, y)\leq  c_2\E_y \wh\tau^B_{D\cap A(z_0, R/4, R/2)}.
 \end{equation}
 In fact, for $u\in D_{R/16}$ and $y\in D\cap A(z_0, 3R/8, R/2),$ we have
\begin{eqnarray}\label{e:2.2'}
G^B_{D_{R/2}}(u, y) &=& \wh{G}^B_{D_{R/2}}(y, u)
=\E_y \wh{G}^B_{D_{R/2}}(\wh X^B_{\wh\tau^B_{D\cap A(z_0, R/4, R/2)}}, u) \nonumber \\
&=& \E_y \Big[\wh{G}^B_{D_{R/2}}(\wh X^B_{\wh\tau^B_{D\cap A(z_0, R/4, R/2)}}, u); \wh X^B_{\wh\tau^B_{D\cap A(z_0, R/4, R/2)}}\in
D_{R/2}\setminus \overline{D_{R/8}} \Big] \nonumber \\
&& \quad+\E_y [\wh{G}^B_{D_{R/2}}(\wh X^B_{\wh\tau^B_{D\cap A(z_0, R/4, R/2)}}, u); \wh X^B_{\wh\tau^B_{D\cap A(z_0, R/4, R/2)}}\in \overline{D_{R/8}}] \nonumber \\
&\leq & \sup_{w\in D_{R/2}\setminus \overline{D_{R/8}}}\wh{G}^B_{D_{R/2}}(w, u)\P_y(\wh X^B_{\wh\tau^B_{D\cap A(z_0, R/4, R/2)}}\in D_{R/2})
\nonumber \\
&& \quad +\E_y\int_0^{\wh\tau^B_{D\cap A(z_0, R/4, R/2)}}\int_{\overline{D_{R/8}}} \wh{G}^B_{D_{R/2}}(w, u)\,\wh J^B(\wh X^B_s, w)\,m_B(dw)\,ds \nonumber \\
&=& \sup_{w\in D_{R/2}\setminus \overline{D_{R/8}}} G^B_{D_{R/2}}(u, w)\P_y(\wh X^B_{\wh\tau^B_{D\cap A(z_0, R/4, R/2)}}\in D_{R/2})
\nonumber \\
&& \quad +\E_y\int_0^{\wh\tau^B_{D\cap A(z_0, R/4, R/2)}}\int_{\overline{D_{R/8}}}
G^B_{D_{R/2}}(u, w)\, J^B(w, \wh X^B_s)\,m_B(dw)\,ds,
\end{eqnarray}
where in the last equality, we used $\wh{G}^B_{D_{R/2}}(w, u)=G^B_{D_{R/2}}(u, w)$ and $\wh J^B(\wh X^B_s, w)m_B(dw)=J^B(w, \wh X^B_s)m_B(dw).$
By assumption (A2), ${\bf (\wh{ED})} _{\leq, z_0, R}$  holds.
Thus there exists $c_3=c_3(z_0, R)>0$  such that for any $y\in D\cap A(z_0, 3R/8, R/2),$
\begin{equation}\label{e:2.5}
\P_y \Big(\wh X^B_{\wh\tau^B_{D\cap A(z_0, R/4, R/2)}}\in D_{R/2} \Big)
\leq c_3\E_y\wh\tau^B_{D\cap A(z_0, R/4, R/2)}.
\end{equation}
By assumption (A3), there exists $c_4=c_4(z_0, R)$ such that
$$
 \sup_{u\in D_{R/16}}\sup_{w\in D_{R/2}\setminus \overline{D_{R/8}}} G^B_{D_{R/2}}(u, w)\leq \sup_{u\in B(z_0, R/16)}\sup_{w\in A(z_0, R/8, R/2)} G^B(u, w)\leq c_4.
 $$
 Hence the first term on the right side of \eqref{e:2.2'} satisfies for $u\in D_{R/16},$
 \begin{equation}\label{e:2.9}
 \sup_{w\in D_{R/2}\setminus \overline{D_{R/8}}} G^B_{D_{R/2}}(u, w)\P_y(\wh X^B_{\wh\tau^B_{D\cap A(z_0, R/4, R/2)}}\in D_{R/2})\leq c_3c_4\E_y \left[ \wh\tau^B_{D\cap A(z_0, R/4, R/2)} \right].
 \end{equation}
For the second term on the right side of \eqref{e:2.2'}, we have for $u\in D_{R/16}$ and $y\in D\cap A(z_0, 3R/8, R/2),$
\begin{equation}\label{e:2.3'}\begin{aligned}
&\E_y\int_0^{\wh\tau^B_{D\cap A(z_0, R/4, R/2)}}\int_{\overline{D_{R/8}}}
G^B_{D_{R/2}}(u, w)\, J^B(w, \wh X^B_s)\,m_B(dw)\,ds\\
&\leq \E_y \left[ \wh\tau^B_{D\cap A(z_0, R/4, R/2)} \right]
 \sup_{ w\in \overline{D_{R/8}}, u_1\in D\cap A(z_0, R/4, R/2)} J^B( {w} , u_1)
\int_{\overline{D_{R/8}}} G^B_{B(z_0, R/2)}(u, w)\, m_B(dw)\\
&\leq \E_y \left[ \wh\tau^B_{D\cap A(z_0, R/4, R/2)} \right] \sup_{ w \in B(z_0, 3R/16), u_1\in A(z_0, R/4, R/2)}
J^B( { w} , u_1) \cdot \E_{u} \tau_{B(z_0, R/2)}\\
&\leq c_5\E_y \left[ \wh\tau^B_{D\cap A(z_0, R/4, R/2)} \right] \sup_{ u_1\in A(z_0, R/4, R/2)} J^B(z_0, u_1),
\end{aligned}\end{equation}
where the last inequality holds due to the condition \eqref{e:J2} of assumption (A4) and Lemma \ref{L:2.1n}.
By the condition\eqref{e:J3} of assumption (A4), we have
\begin{equation}\label{e:2.10}
\sup_{u_1\in A(z_0, R/4, R/2)} J^B(z_0, u_1):=c_6(z_0, R)<\infty.
\end{equation}
By \eqref{e:2.10} together with  \eqref{e:2.3'}, we have  for the second term on the right side of \eqref{e:2.2'},
\begin{equation}\label{e:2.12'}
\E_y\int_0^{\wh\tau^B_{D\cap A(z_0, R/4, R/2)}}\int_{\overline{D_{R/8}}}
G^B_{D_{R/2}}(u, w)\, J^B(w, \wh X^B_s)\,m_B(dw)\,ds
\leq c_5c_6\E_y \left[ \wh\tau^B_{D\cap A(z_0, R/4, R/2)} \right].
\end{equation}
Consequently, by \eqref{e:2.2'}, \eqref{e:2.9} and \eqref{e:2.12'},  we establish \eqref{e:2.3}.

By  \eqref{e:2.2} and \eqref{e:2.3}, we have for $x\in D_{R/64}$ and $y\in D\cap  A(z_0, 3R/8, R/2),$
\begin{equation}\label{e:2.12n}\begin{aligned}
I_1(x, y)&\leq c_1c_2\E_x \left[  \tau_{D_{R/32}} \right] \E_y  \left[ \wh\tau^B_{D\cap A(z_0, R/4, R/2)} \right] \\
&= c_1c_2\E_x \left[  \tau_{D_{R/32}} \right]
\int_{D\cap A(z_0, R/4, R/2)} \hat G^B_{D\cap A(z_0, R/4, R/2)}(y, w)\, m_B(dw)\\
&\leq c_1c_2 \E_x  \left[ \tau_{D_{R/32}} \right] \int_{D_{R/2}\setminus D_{R/16}} \hat G^B_{D_{R/2}}(y, w)\, m_B(dw)\\
&=c_1c_2\E_x \left[ \tau_{D_{R/32}} \right] \int_{D_{R/2}\setminus D_{R/16}} G^B_{D_{R/2}}(w, y)\, m_B(dw).
\end{aligned}\end{equation}

By the L\'evy system formula of $X$ and the conditions \eqref{e:J3} and \eqref{e:J2} of assumption (A4),  for $x\in D_{R/64}$  and $y\in D\cap  A(z_0, 3R/8, R/2),$
\begin{equation}\label{e:2.13n}\begin{aligned}
I_2(x, y)&=\E_x\int_0^{\tau_{D_{R/32}}}\int_{D_{R/2}\setminus D_{R/16}} G^B_{D_{R/2}}(w, y) J^B(X_s, w)\,m_B(dw)\,ds\\
&\leq \sup_{u\in D_{R/32}, w\in D_{R/2}\setminus D_{R/16}} J^B(u, w) \E_x \left[ \tau_{D_{R/32}} \right] \int_{D_{R/2}\setminus D_{R/16}} G^B_{D_{R/2}}(w, y)\, m_B(dw)\\
&\leq c_8\sup_{w\in B(z_0, R/2)\setminus B(z_0, R/16)} J^B(z_0, w) \E_x \left[ \tau_{D_{R/32}} \right] \int_{D_{R/2}\setminus D_{R/16}} G^B_{D_{R/2}}(w, y)\, m_B(dw)\\
&\leq c_9\E_x \left[ \tau_{D_{R/32}} \right] \int_{D_{R/2}\setminus D_{R/16}} G^B_{D_{R/2}}(w, y)\, m_B(dw),
\end{aligned}\end{equation}
where $c_k=c_k(z_0, R)\in (0, \infty), k=8, 9.$
By \eqref{e:2.4}, \eqref{e:2.12n} and \eqref{e:2.13n}, for $x\in D_{R/64}$ and $y\in D\cap A(z_0, 3R/8, R/2),$
\begin{equation}\label{e:2.9'}
G^B_{D_{R/2}}(x, y)=I_1(x, y)+I_2(x, y)\leq (c_1c_2+c_9)\E_x \left[ \tau_{D_{R/32}} \right]\int_{D_{R/2}\setminus D_{R/16}} G^B_{D_{R/2}}(w, y)\, m_B(dw).
\end{equation}
That is, the upper bound of \eqref{e:2.2n} holds.

On the other hand, by the L\'evy system formula of $X$ and the condition \eqref{e:J2} of assumption (A4),  for $x\in D_{R/64}$ and  $y\in D\cap A(z_0, 3R/8, R/2),$
\begin{equation}\label{e:2.14n}\begin{aligned}
G^B_{D_{R/2}}(x, y)&=\E_x [G^B_{D_{R/2}}(X_{\tau_{D_{R/32}}}, y)]\\
&\geq \E_x [G^B_{D_{R/2}}(X_{\tau_{D_{R/32}}}, y); X_{\tau_{D_{R/32}}}\in D_{R/2}\setminus D_{R/16}]\\
&=\E_x\int_0^{\tau_{D_{R/32}}}\int_{D_{R/2}\setminus D_{R/16}} G^B_{D_{R/2}}(w, y) J^B(X_s, w)\,m_B(dw)\,ds\\
&\geq \inf_{u\in D_{R/32}, w\in D_{R/2}\setminus D_{R/16}} J^B(u, w) \E_x \left[ \tau_{D_{R/32}} \right] \int_{D_{R/2}\setminus D_{R/16}} G^B_{D_{R/2}}(w, y)\, m_B(dw)\\
&\geq c_{10}\inf_{w\in B(z_0, R/2)\setminus B(z_0, R/16)} J^B(z_0, w)\cdot \E_x \left[ \tau_{D_{R/32}} \right]\int_{D_{R/2}\setminus D_{R/16}} G^B_{D_{R/2}}(w, y)\, m_B(dw).
\end{aligned}\end{equation}
Note that by the condition  \eqref{e:J3} of assumption (A4),
$$
c_{11}(z_0, R):=\inf_{w\in B(z_0, R/2)\setminus B(z_0, R/16)} J^B(z_0, w)>0.
$$
Hence by this together with \eqref{e:2.14n}, the lower bound of \eqref{e:2.2n} holds.
Consequently, \eqref{e:2.2n} holds.

 Furthermore, if  assumptions (A2')-(A4') and (A5) hold, then  the constants in (A2)-(A4)
hold uniformly for $x_0\in \sX$.
By assumption (A5), for each $R\in (0, R_0/8),$ one has
\begin{equation}\label{e:2.15}
\sup_{x\in B(x_0, R/2), x_0\in \sX} \E_x \left[ \tau_{B(x_0, R)} \right]\leq \sup_{x\in \sX} \E_x \left[ \tau_{B(x, 2R)} \right]\leq c
\end{equation}
 for some constant $c=c(R)>0$. Hence under this assumption, the constant $C$ in \eqref{e:2.1} of Lemma \ref{L:2.1n} holds uniformly for $x_0\in \sX$.
  By carefully tracking all
  the constants in the proof,  one finds that
  the constant $C$ in \eqref{e:2.2n}  depends only on $R.$ The proof is complete.
\qed

\smallskip

By Lemma \ref{L:3.2}, we immediately obtain the following proposition.

\begin{prp}\label{P:2.3}
Suppose assumptions (A1) and (A2)-(A4) hold for some $R_0>0$.
 Let $D$ be an  open subset of $\sX$. For each
 $B=B(z_0, R)$ with $z_0\in \partial D$ and  $R\in (0, R_0/8)$ so that $\overline {B}\subsetneq \sX,$
there exists $C=C(z_0, R)>0$ depending on the constants in (A2)-(A4) for the ball $B(z_0, R)$
such that for any
$x_k\in D\cap B(z_0,  R/64)$ with $k=1,2$ and $y\in D\cap  A(z_0, 3R/8, R/2), $
\begin{equation}\label{e:3.2'}\begin{aligned}
&C^{-1}G^B_{D\cap B(z_0, R/2)}(x_2, y) \dfrac{\E_{x_1} \left[ \tau_{D\cap B(z_0, R/32)} \right]}{\E_{x_2} \left[ \tau_{D\cap B(z_0, R/32)}  \right]}\leq G^B_{D\cap B(z_0, R/2)}(x_1, y)\\
&\qquad  \leq C \dfrac{\E_{x_1} \left[ \tau_{D\cap B(z_0, R/32)} \right]}{\E_{x_2} \left[ \tau_{D\cap B(z_0, R/32)}  \right]}G^B_{D\cap B(z_0, R/2)}(x_2, y).
\end{aligned}\end{equation}
 If we further assume  (A2')-(A4') and (A5) hold,  then the constant in \eqref{e:3.2'} depends only on $R.$
\end{prp}

\begin{lem}\label{L:2.5}
Suppose assumptions (A2) and (A4) hold for some $R_0>0$.
Let $D$ be an  open subset of $\sX$.
Then for each $z_0\in \partial D$ and $R\in (0, R_0/16),$    there exists $C=C(z_0, R)>1$   such that
\begin{equation}\label{e:2.12}
\E_x\tau_{D\cap B(z_0, R)}\leq C \, \E_x\tau_{D\cap B(z_0, R/2)}, \quad x\in D\cap B(z_0, R/4).
\end{equation}
 If  assumptions (A2') and (A5) hold, then the multiplicative constant $C$ in \eqref{e:2.12} depends only on $R.$
  \end{lem}

 \pf  Let $z_0\in \partial D$ and $R\in (0, R_0/16).$ For notational  simplicity, for each $R>0,$ let $D_R(z_0):=D\cap B(z_0, R).$
By the strong Markov property of $X,$ for  $x\in D_{R/4}(z_0)$,
$$
\E_x\tau_{D_R(z_0)}= \E_x \tau_{D_{R/2}(z_0)}+ \E_x  \left[\E_{X_{\tau_{D_{R/2}(z_0)}}} \tau_{D_R(z_0)} ; X_{\tau_{D_{R/2}(z_0)}}\in D_R(z_0) \right].
 $$
Note that $R\in (0, R_0/16).$ Then by assumption (A2) with $R_0>0,$   ${\bf (ED)} _{\leq, z_0, 16R}$   holds.
We have for $x\in D_{R/4}(z_0),$
\begin{eqnarray*}
\E_x\tau_{D_R(z_0)}
 &\leq &\E_x \tau_{D_{R/2}(z_0)}+ \E_x  \left[ \sup_{z\in D_R(z_0)} \E_{z} \tau_{D_R(z_0)} ; X_{\tau_{D_{R/2}(z_0)}}\in D_R(z_0) \right] \\
&\leq & \E_x \tau_{D_{R/2}(z_0)}+\P_x(X_{\tau_{D_{R/2}(z_0)}}\in D_{8R}(z_0))  \, \sup_{z\in B(z_0, R)} \E_{z} \tau_{B(z_0, 2R)}  \\
&\leq & \E_x \tau_{D_{R/2}(z_0)}+ c_1\E_x \tau_{D_{R/2}(z_0)}  \\
 &\leq & (c_1+1) \E_x \tau_{D_{R/2}(z_0)},
\end{eqnarray*}
where $c_1=c_1(z_0, R)$ and   Lemma \ref{L:2.1n} is used in the third line. Hence \eqref{e:2.12} holds.

 If assumptions (A2') and  (A5) hold, then for each $R\in (0, R_0/16),$ ${\bf (ED)^s} _{\leq, 16R}$   holds
and by \eqref{e:2.15}, $\sup_{z\in B(z_0, R)} \E_{z} \tau_{B(z_0, 2R)}\leq c$ for a positive constant $c=c(R)>0$ independent of $z_0\in \partial D.$
Thus in this case, the  constant $c_1$ above depends only on $R.$ Hence the multiplicative constant in \eqref{e:2.12} depends only on $R.$
\qed

\smallskip

 As we only assume the weak duality assumption in each ball $B$ with respect to the reference measure $m_B$ in assumption (A1), in the following proposition, we adapt the representation formula of regular harmonic functions for discontinuous Hunt process in open sets  vanishing in part of the boundary of open sets in Chen and Wang \cite{CW2} to the subprocess $X^B.$

\begin{prp}\label{P:2.2}
Suppose assumption (A1) holds. Let $B$ be a ball in $\sX$
so that $\overline{B}\subsetneq \sX$,
and  $D$ be an open subset of  $\sX$.  Let  $U$ and $V$ be relatively compact  open sets in $B$ with $\overline V  \subseteq U$ and $D\cap V\neq \varnothing.$
  Suppose $h$ is a nonnegative  regular   harmonic function with respect to $X^B$in $D\cap U$ and
    vanishes on $(D^c)^r\cap U$.
 Then
 there exist a sequence of Radon measures $\{\mu_n; n\geq 1\}$ on $(D\cap U)\setminus V$
  so that     $G^B_{D\cap U}\mu_n(x):= \int_{(D\cap U)\setminus V} G^B_{D\cap U}(x, y) \mu_n (dy)$   is uniformly bounded on $D\cap V$ and increasing   in $n\geq 1$,  and for $x\in   D\cap V,$
  \begin{eqnarray}\label{e:2.5a}
h(x)& = &  \lim_{n\to \infty} \int_{(D\cap U)\setminus  V} G^B_{D\cap U}(x, y)\mu_n(dy) \nonumber\\
&&  + \int_{D\cap U} G^B_{D\cap U}(x,y)\int_{B\setminus U}h(z) J(y, dz)\,m_B(dy).
\end{eqnarray}
\end{prp}

\pf
 Since $X$ is irreducible  and   $\overline{B}\subsetneq \sX$, by  Lemma 2.3 in \cite{CW2}, the semigroup of $X^B$  is transient  in the sense of
  \cite[Definition on page 86]{ChungW} that there exists a strictly positive function $g_0$ on $B$ such that  $0<G^B g_0(x) <\infty$  for every $x\in B.$
  It is well known  (see e.g. \cite[Proposition (3.3) on p.80]{BG})   that  $ D^c\setminus (D^c)^r$ is semi-polar for the process $X$.
Hence it is polar by  Hunt's hypothesis; that is,
$\P_y(X_t  \mbox{ ever hits } D^c\setminus (D^c)^r)=0$ for all $y\in \sX.$
Thus, $\P_y(X^B_t  \mbox{ ever hits } D^c\setminus (D^c)^r)=0$ for all $y\in B.$
Note that the subprocess $X^B$ has a L\'evy system $(J^B(x, dy), dt),$ where
$J^B(x, dy)=J(x, dy)$ on $B\times B.$
Then by the argument above,  assumption (A1) and  an argument similar to that of Proposition 2.4 in \cite{CW2}  but with $X^B$ in place of $X$ there,   we get the desired conclusion.
  \qed

\begin{proof}[Proof of Theorem \ref{T0}.]
Suppose that assumptions (A1) and (A2)-(A4)  hold for some $R_0>0$.
Let $D$ be an  open subset of $\sX$.
Let $z_0\in \partial D$  and $R\in (0, \tfrac{1}{16}R_0)$
so that $\overline{B(z_0, 2R)}\subsetneq \sX$. Let $B:=B(z_0, R).$
For each $\lambda >0,$  we use $\lambda B$ to denote the concentric ball $B(z_0, \lambda R).$
For the simplicity of notation, for each $s>0,$ let $D_s:=D\cap B(z_0, s).$
 Let  $h$ be a  nontrivial   nonnegative  regular harmonic function in $D_R$ with respect to $X$ vanishing on $(D^c)^r\cap B.$

 Let
\begin{equation}
\tilde h(x):=\E_x  \left[h(X_{\tau_{D_R}}); X_{\tau_{D_R}}\in 2B \right]=\E_x  \left[h(X_{\tau_{D_R}}); \tau_{2B}>\tau_{D_R} \right] \quad
 \hbox{for } x\in 2B.
\end{equation}
Note that $D_R\subset B$, so $\tilde h$ is a non-negative regular harmonic function in $D_R$ with respect to the part process $X^{2B}$ vanishing on $(D^c)^r\cap B.$
By applying Proposition \ref{P:2.2} with $X^{2B}, B$ and $\tfrac{7}{8}B$ in place of $X^B, U$ and $V$ there, there exist a sequence of Radon measures $\{\mu_n; n\geq 1\}$ on $D_R\setminus D_{7R/8}$
  so that
  $$
  G^{2B}_{D_R}\mu_n(x):= \int_{D_R\setminus D_{7R/8}} G^{2B}_{D_R}(x, y) \mu_n (dy)
  $$
    is uniformly bounded on $D_{7R/8}$ and increasing   in $n\geq 1$,
and for $x\in   D_{R/32},$
  \begin{eqnarray}\label{e:2.6a}
\tilde h(x)& = &  \lim_{n\to \infty} \int_{D_R\setminus D_{7R/8}} G^{2B}_{D_R}(x, y)\mu_n(dy) \nonumber\\
&&  + \int_{D_R} G^{2B}_{D_R}(x,y)\int_{(2B)\setminus B}\tilde h(z) J(y, dz)\,m_{2B}(dy)\nonumber\\
& = &  \lim_{n\to \infty} \int_{D_R\setminus D_{7R/8}} G^{2B}_{D_R}(x, y)\mu_n(dy) \nonumber\\
&&+\int_{D_R\setminus D_{7R/8}}\int_{(2B)\setminus B} G^{2B}_{D_R}(x, y)  \tilde h(z)J^{2B}(y, z) \, m_{2B}(dz)\,m_{2B}(dy) \nonumber\\
&&+\int_{D_{7R/8}}\int_{(2B)\setminus B} G^{2B}_{D_R}(x, y)  \tilde h(z)J^{2B}(y, z) \, m_{2B}(dz)\,m_{2B}(dy) \nonumber\\
&:=& \tilde h_1(x)+\tilde h_2(x)+\tilde h_3(x).
\end{eqnarray}

By  Proposition \ref{P:2.3} but with $2B$ and $2R$ in place of $B$ and $R$ there,
 there exists $c_1=c_1(z_0, R)>1$ such that for any $x_1, x_2\in D_{R/32},$
\begin{equation}\label{e:2.11}
c_1^{-1}\dfrac{\E_{x_1} \tau_{D_{R/16}}}{\E_{x_2} \tau_{D_{R/16}}} \, \tilde h_k(x_2)\leq \tilde h_k(x_1)\leq c_1\dfrac{\E_{x_1} \tau_{D_{R/16}}}{\E_{x_2} \tau_{D_{R/16}}} \, \tilde h_k(x_2), \quad k=1, 2.
\end{equation}
By the condition \eqref{e:J2} of assumption (A4) for $R_0$, there exists $c_2=c_2(z_0, R)>0$ such that for $x\in D_{R/32},$
$$\begin{aligned}
\tilde h_3(x)&=\int_{D_{7R/8}}\int_{(2B)\setminus B} G^{2B}_{D_R}(x, y) \tilde h(z)J^{2B}(y, z)\, m_{2B}(dz)\,m_{2B}(dy)\\
&\leq c_2\int_{(2B)\setminus B}\Big[\int_{D_{7R/8}} G^{2B}_{D_R}(x, y)\,m_{2B}(dy)\Big] \tilde h(z)J^{2B}(z_0, z)\, m_{2B}(dz)\\
&\leq c_2\E_x \tau_{D_R}\int_{(2B)\setminus B}\tilde h(z)J^{2B}(z_0, z)m_{2B}(dz).
\end{aligned}$$
On the other hand, by the condition \eqref{e:J2} of assumption (A4), there exists $c_3=c_3(z_0, R)>0$ such that for $x\in D_{R/32},$
$$\begin{aligned}
\tilde h_3(x)&=\int_{D_{7R/8}}\int_{(2B)\setminus B} G^{2B}_{D_R}(x, y) \tilde h(z)J^{2B}(y, z)\, m_{2B}(dz)\,m_{2B}(dy)\\
&\geq c_3\int_{(2B)\setminus B}\Big[\int_{D_{7R/8}} G^{2B}_{D_R}(x, y)\,m_{2B}(dy)\Big] \tilde h(z)J^{2B}(z_0, z)\, m_{2B}(dz)\\
&\geq c_3\E_x \tau_{D_{7R/8}}\int_{(2B)\setminus B}\tilde h(z)J^{2B}(z_0, z)\, m_{2B}(dz).
\end{aligned}$$
By  Lemma \ref{L:2.5}, there exists $c_4=c_4(z_0, R)>0$ such that for $x\in D_{R/32},$
$$\E_x \tau_{D_{7R/8}}\leq \E_x \tau_{D_R}\leq c_4\E_x \tau_{D_{R/2}}\leq c_4\E_x \tau_{D_{7R/8}}.$$
Hence, for any $x_1, x_2\in D_{R/32},$
\begin{equation}\label{e:2.13}
\dfrac{c_3}{c_2c_4}\dfrac{\E_{x_1} \tau_{D_R}}{\E_{x_2} \tau_{D_R}}\tilde h_3(x_2)\leq \tilde h_3(x_1)\leq \dfrac{c_2c_4}{c_3}\dfrac{\E_{x_1} \tau_{D_R}}{\E_{x_2} \tau_{D_R}}\tilde h_3(x_2).
\end{equation}
By using Lemma \ref{L:2.5} again,
there exists $c_5=c_5(z_0, R)>0$ such that for $x\in D_{R/32},$
$$\E_x \tau_{D_{R/16}}\leq \E_x \tau_{D_R}\leq c_5\E_x \tau_{D_{R/16}}.$$
By combining this inequality and \eqref{e:2.6a}-\eqref{e:2.13}, we get there exists $c_6=c_6(z_0, R)>1$ such that for $x_1, x_2\in D_{R/32},$
\begin{equation}\label{e:2.14}
c_6^{-1}\dfrac{\E_{x_1} \tau_{D_R}}{\E_{x_2} \tau_{D_R}}\tilde h(x_2)\leq \tilde h(x_1)\leq c_6\dfrac{\E_{x_1} \tau_{D_R}}{\E_{x_2} \tau_{D_R}}\tilde h(x_2).
\end{equation}

By the L\'evy system formula of $X$,
$$
g(x):=h(x)-\tilde h(x)=\E_x  \left[h(X_{\tau_{D_R}}); X_{\tau_{D_R}}\in (2B)^c \right]=\E_x\int_0^{\tau_{D_R}}\int_{(2B)^c} h(z)J(X_s, dz)\,ds.
$$
By the condition \eqref{e:J1} of assumption (A4), there exists $c_7=c_7(z_0, R)>1$ such that for $u\in D_R,$
$$c^{-1}_7 J(z_0, dz)\leq J(u, dz)\leq c_7 J(z_0, dz) \quad \mbox{on} \: (2B)^c.$$
Hence for $x\in D_{R/32},$
$$\begin{aligned}
c^{-1}_7\E_x\tau_{D_R}\int_{(2B)^c} h(z)J(z_0, dz)\leq g(x)\leq c_7\E_x\tau_{D_R}\int_{(2B)^c} h(z)J(z_0, dz).
\end{aligned}$$
Thus for $x_1, x_2\in D_{R/32},$
\begin{equation}\label{e:2.22}
c_7^{-2}\dfrac{\E_{x_1} \tau_{D_R}}{\E_{x_2} \tau_{D_R}}g(x_2)\leq g(x_1)\leq c_7^2\dfrac{\E_{x_1} \tau_{D_R}}{\E_{x_2} \tau_{D_R}}g(x_2).
\end{equation}
Note that $h(x)=\tilde h(x)+g(x).$ Then by \eqref{e:2.14} together with \eqref{e:2.22},  there exists $c_8=c_8(z_0, R)>1$ such that for $x_1, x_2\in D_{R/32},$
\begin{equation}\label{e:2.23}
c_8^{-1}\dfrac{\E_{x_1} \tau_{D_R}}{\E_{x_2} \tau_{D_R}}h(x_2)\leq h(x_1)\leq c_8\dfrac{\E_{x_1} \tau_{D_R}}{\E_{x_2} \tau_{D_R}}h(x_2).
\end{equation}
The BHP \eqref{e:1} follows and the constant in \eqref{e:1} depends on the constants in (A2)-(A4).

 Furthermore, assume  that (A1), (A2')-(A4') and (A5) hold.  Then by Proposition \ref{P:2.3}, Lemma \ref{L:2.5} and by carefully tracking the constants in the proof, the constant in \eqref{e:2.23} depends only on $R.$ The proof is complete.
\end{proof}

\begin{remark} \rm
By \eqref{e:2.23}, for any  nonnegative  regular harmonic function $h$
on $D\cap B(z_0, R)$   vanishing  on  $(D^c)^r\cap B(z_0, R),$  the boundary decay rate for $h$ at $\partial D\cap B(z_0,  R)$ is $\E_x \tau_{D\cap B(z_0, R)}$.

\end{remark}

\section{ Exit distribution estimates}\label{S:3}

In this section, we will prove Proposition \ref{P1},  which provides some sufficient conditions  for
 assumption (A2).
  In view of \cite[Proposition 3.1]{CC},  ${\rm \bf (Jt)}_{\leq, R, \gamma}$  and  ${\bf (EP)} _{\leq, R, \gamma}$  for some $\gamma>0$
  yields  ${\bf (ED)}^s _{\leq, R}$ holds.
 Hence it suffices to show
  that for each $R>0,$
  if {\rm ${\bf (\wh{EP})} _{\leq , R, \gamma}$} and
${\bf (\wh{Jt})}_{\leq, R, \gamma}$ hold   for some
   $\gamma>0$,  then ${\bf (\wh{ED})}^s _{\leq, R}$   holds.
Recall that for each $\xi\in  \sX, R>0$ and open set $U,$ define $B_U(\xi, R):=U\cap B(\xi, R).$

\begin{lem}\label{L:3.1'}
 Suppose ${\bf (\wh{EP})} _{\leq , r_0, \gamma}$ holds  with  $r_0>0$ and $\gamma>0.$
  Then there is a constant $C=C(r_0, \gamma)\geq 1$ such that
  for every ball $B$ with radius $r_0$ and  $\overline B\subsetneq \sX$,
   each open $U\subset  2^{-1}B$, $\xi\in U$ and $r\in (0, r_0/4)$,
$$
\P_x \left(\wh X^B_{\wh\tau^B_{B_U (\xi, r)}} \in U \right)
=\P_x \Big(\wh \tau^B_{U}>\wh\tau^B_{B_U(\xi,r)}\Big)
\leq   C \left( \E_x \left[ \wh\tau^B_{B_U(\xi,r)} \right] / r^\gamma \right)^{1/2} \quad\hbox{ for } x\in B_U(\xi, 3r/4).
$$
\end{lem}

\pf The  proof is similar to Lemma 3.2 in  \cite{CC}.
Let $B$ be a ball with radius $r_0$ so that $\overline B\subsetneq \sX.$
Let  $U\subset 2^{-1}B$ be an open subset and $\xi\in U.$
 We fix $x\in B_U(\xi,3r/4)$  and let $r_1=r/4$.
Suppose the condition  ${\bf (\wh{EP})} _{\leq , r_0, \gamma}$ holds, then there exists $c_1=c_1(r_0, \gamma)$  such that for  $t>0,$
\begin{equation}\label{e:3.2n}\begin{aligned}
&\P_x\big(\wh\tau^B_U>\wh\tau^B_{B_U(x,r_1)}\big)=\P_x\big(\wh\tau^B_U>\wh\tau^B_{B(x,r_1)}\big)\\
=&\P_x\big(\wh\tau^B_U>\wh\tau^B_{B(x,r_1)}, \wh\tau^B_{B(x, r_1)}<t\big)+\P_x\big(\wh\tau^B_U>\wh\tau^B_{B(x,r_1)}\geq t\big)\\
\leq & \P_x\big(\wh\tau^B_{B(x, r_1)}<t\wedge \wh\tau^B_{2^{-1}B} \big)+\P_x(\wh\tau^B_{B_U(x,r_1)}\geq t)\\
\leq &c_1\frac{t}{r_1^\gamma}+\Big(1\wedge\frac{\mathbb{E}_x[\wh\tau^B_{B_U(x, r_1)}]}{t}\Big).
\end{aligned}\end{equation}
Note that $B_U(x, r_1)\subset B_U(\xi, r)$ for $x\in B_U(\xi, 3r/4)$.
Taking $t=\sqrt{r_1^\gamma\E_x[\wh\tau^B_{B_U(\xi,r)}]}$ in \eqref{e:3.2n} yields that for every $x\in B_D(\xi, 3r/4),$
$$
\P_x\big(\wh\tau^B_{U}>\wh\tau^B_{B_U(\xi,r)}\big)
\leq \P_x\big(\wh\tau^B_{U}>\wh\tau^B_{B_U(x,r/4)}\big)
\leq   c_2\Big(\dfrac{\E_x\wh\tau^B_{B_U(\xi, r)}}{r^\gamma}\Big)^{1/2}.
$$
 Thus the desired conclusion is obtained.
\qed

\medskip

By Lemma \ref{L:3.1'} and an argument similar to that of Corollary 3.3 in \cite{CC}, we have the following result.

 \begin{cor}\label{C:1}
  Suppose ${\bf (\wh{EP})} _{\leq , r_0, \gamma}$ holds with  $r_0>0$ and $\gamma>0.$ There is a constant $C=C(r_0, \gamma)\geq 1$ such that
    for every ball $B$ with radius $r_0$ and $\overline B\subsetneq \sX$,
    each open $U\subset  2^{-1}B$, $\xi\in U$ and $r\in (0, r_0)$,
\begin{equation}\label{e:3.6}
\frac{\E_x[\wh \tau^B_{B_U(\xi,r)}]}{r^\gamma\P_x\big(\wh \tau^B_U>\wh \tau^B_{B_U(\xi,r)}\big)}
\geq C\left(\P_x\big(\wh \tau^B_U>\wh \tau^B_{B_U(\xi,r)}\big)+\frac{\E_x[\wh \tau^B_{B_U(\xi,r)}]}{r^\gamma}\wedge1\right) \quad\hbox{ for }x\in B_U(\xi, 3r/4).
\end{equation}
\end{cor}

\begin{lem}\label{L:3.4}
Suppose that ${\rm \bf (\wh{Jt})}_{\leq, \bar r, \gamma}$ holds  with   $\bar r>0$ and $\gamma>0.$ Then there is a constant $C=C(\bar r, \gamma)\in  (0,\infty)$
such that for every ball $B$ with radius $\bar r$ and  $\overline B\subsetneq \sX$,
  every open $U\subset 2^{-1}B$ and every $W\subset \tfrac{1}{2}B\setminus \overline{U}$, we have
\begin{equation}\label{e:3.4}
\P_x(\wh X^B_{\wh\tau^B_U}\in W)\leq C\frac{\E_x[\wh\tau^B_U]}{(d(U,W) \wedge\bar{r})^\gamma} \quad \hbox{ for every }x\in U.
\end{equation}
 Here $d(U, W)$ denotes the distance between two sets $U$ and $W$, that is,
$d(U, W)=\inf\{d(x, y): x\in U \hbox{ and } y\in W\}$.
\end{lem}

\pf Let $B$ be a ball with radius $\bar r$ so that $\overline B\subsetneq \sX.$
  By using the L\'evy system of $\wh X^B$, and by the condition $ {\bf (\wh{Jt})}_{\leq, \bar r, \gamma}$, we have for $x\in U,$
$$\begin{aligned}
\P_x(\wh X^B_{\wh\tau_U}\in W)&=\E_x \int_0^{\wh\tau^B_U}\wh J^B(\wh X^B_t,W)dt\leq \E_x \int_0^{\wh\tau^B_U}\wh J^B(\wh X^B_t,B(\wh X^B_t,d(U,W)\wedge\tfrac{\bar{r}}{4})^c\cap \tfrac{1}{2}B)dt\\
&\leq \E_x \int_0^{\wh\tau^B_U} \frac{C}{(d(U,W)\wedge\bar{r})^\gamma}dt =C\frac{\E_x[\wh\tau^B_U]}{(d(U,W)\wedge\bar{r})^\gamma}.
\end{aligned}$$
\qed

\begin{prp}\label{P:3.5}
  Suppose that   {\rm ${\bf (\wh{EP})} _{\leq , R, \gamma}$} and
${\bf (\wh{Jt})}_{\leq, R, \gamma}$ hold for $R>0$  with   some $\gamma>0$.  Then there is a constant
 $C=C(R, \gamma)\geq 1$ such that
 for every ball $B$ with radius $R$ and  $\overline B\subsetneq \sX$,
for any open subset
$ U\subset 2^{-1}B$, $\xi \in U$ and  $r\in (0, R/4)$,
$$
\P_x(\wh \tau^B_U>\wh \tau^B_{B_U(\xi,r)})=\P_x (\wh X^B_{\wh\tau^B_{B_U (\xi, r)}} \in U)\leq   C\dfrac{\E_x[\wh \tau^B_{B_U(\xi, r)}]}{r^\gamma} \quad\hbox{ for every }x\in B_U(\xi, r/2).
$$
\end{prp}

\pf  The proof is similar to that of \cite[Theorem 3.1]{CC}, using a ``box" method argument.
For the reader's convenience, we spell out the details.
Let $B$ be a ball  with radius $R$ so that $\overline B\subsetneq \sX.$
Let $ U\subset 2^{-1}B$ be an open subset, $\xi \in U$ and  $r\in (0, R/4).$
  Recall that $\sum_{k=1}^\infty k^{-2}= \pi^2/6$.
First, we introduce some subsets of $B_U(\xi,3r/4)$ as follows. Define
  $$
U_0  :=\Big\{x\in B_U(\xi,3r/4): \P_x(\wh \tau^B_U>\wh \tau^B_{B_U(\xi,r)})+r^{-\gamma}\E_x[\wh \tau^B_{B_U(\xi,r)}]\geq \frac12\Big\},
$$
and for $j\geq 1$,
\begin{align*}
U_j& :=\Big\{x\in B_U \Big( \xi,\frac{3r}4-\sum_{k=1}^j\frac{3r}{2\pi^2k^2} \Big): \P_x(\wh \tau^B_U>\wh \tau^B_{B_U(\xi,r)})
+r^{-\gamma}\E_x[\wh \tau^B_{B_U(\xi,r)}]
\in [2^{-(j+1)},2^{-j}) \Big\} ,\\
V_j&:=\bigcup_{k=0}^{j-1}U_k \qquad \hbox{and} \qquad
W_j :=B_U \Big( \xi,\frac{3r}4 -\sum_{k=1}^{j-1}\frac{3r}{2\pi^2k^2} \Big) \setminus V_j.
\end{align*}
Observe that  for $j\geq 0$,
\begin{align}\label{e:3.1}
\begin{split}
V_j &  \supset   \Big\{x\in B_U \Big( \xi,\frac{3r}4-\sum_{k=1}^{j-1}\frac{3r}{2\pi^2k^2} \Big): \P_x(\wh \tau^B_U>\wh \tau^B_{B_U(\xi,r)})
+r^{-\gamma}\E_x[\wh \tau^B_{B_U(\xi,r)}]
  \geq 2^{-j} \Big\}, \\
W_j &=  \Big\{x\in B_U \Big( \xi,\frac{3r}4-\sum_{k=1}^{j-1}\frac{3r}{2\pi^2k^2} \Big): \P_x(\wh \tau^B_U>\wh \tau^B_{B_U(\xi,r)})
+r^{-\gamma} \E_x[\wh \tau^B_{B_U(\xi,r)}]     < 2^{-j} \Big\},
\end{split}
\end{align}
and
$$
B_U(\xi, r/2)\subset \bigcup_{j=1}^\infty V_j=\bigcup_{j=0}^\infty U_j \subset B_U(\xi, 3r/4).
$$
For $j\geq 1$, we define
\begin{equation}\label{eqn32}
\lambda_j=\begin{cases}
	\inf\limits_{x\in V_j}\frac{\E_x[\wh \tau^B_{B_U(\xi,r)}]}{\P_x(\wh \tau^B_U>\wh \tau^B_{B_U(\xi,r)})r^\gamma}
	 &\hbox{ if }V_j\neq\emptyset,\\
	+\infty        &\hbox{ if }V_j=\emptyset.
\end{cases}
\end{equation}
For each $j\geq 1$ and $x\in U_j$, define $\wh \tau^B_{x,j}:=\wh \tau^B_{B_U(x,\frac{3r}{4\pi^2j^2})}$.

 \medskip

\noindent(i) By \eqref{eqn32} and the strong Markov property of $\wh X^B$, we have  for $x\in U_j$,
\begin{eqnarray}\label{eqn33}
 \P_x(\wh X^B_{{ \wh \tau^B_{x,j} } }\in V_j,\wh \tau^B_U>\wh \tau^B_{B_U(\xi,r)})
 &=& \E_x[\P_{\wh X^B_{{ \wh \tau^B_{x,j} } }}(\wh \tau^B_U>\wh \tau^B_{B_U(\xi,r)});\wh X^B_{{ \wh \tau^B_{x,j} } }\in V_j]
 \nonumber \\
&\leq &  \lambda_j^{-1} r^{-\gamma}\E_x \left[   {\E_{\wh X^B_{{ \wh \tau^B_{x,j} } }}[\wh \tau^B_{B_U(\xi,r)}]}  ;  \wh X^B_{{ \wh \tau^B_{x,j} } }\in V_j  \right]
 \nonumber \\
 &\leq &  \lambda_j^{-1}\E_x[\wh \tau^B_{B_U(\xi,r)}-{ \wh \tau^B_{x,j} } ]/r^\gamma.
 \end{eqnarray}

\medskip

\noindent(ii) Since $x\in U_j\subset B_U(\xi,\frac{3r}4-\frac{3r}{2\pi^2})$,  $B_U(x,\frac{3r}{4\pi^2j^2})\subset B_U(\xi,\frac{3r}{4} -\frac{3r}{4\pi^2})$ and so the distance between $B_U(x,\frac{3r}{4\pi^2j^2})$ and $B_U(\xi, {3r}/4)^c$ is at least $\frac{3r}{4\pi^2}$. Thus
 by  Lemma \ref{L:3.4},
\begin{equation}\label{eqn34}\begin{aligned}
\P_x\big(\wh X^B_{{ \wh \tau^B_{x,j} } }\in U \setminus B_U(\xi, {3r}/4)\big)\leq c_1\frac{\E_x[{ \wh \tau^B_{x,j} } ]}{(\frac{3r}{4\pi^2})^\gamma}\leq c_2\frac{\E_x[{ \wh \tau^B_{x,j} } ]}{r^\gamma} .
\end{aligned}\end{equation}

\medskip

\noindent(iii) For $2\leq i\leq j$  and $x\in U_j\subset B_U(\xi,\frac{3r}4-\sum_{k=1}^j\frac{3r}{2\pi^2k^2})$, noticing that
\begin{eqnarray*}
d\Big(B_U \big(x,\frac{3 r}{4\pi^2j^2}  \big),W_{i-1}\setminus (V_i\cup W_i )\Big)
& \geq &d\Big(x,  W_{i-1}\setminus B\big(\xi,\frac{3r}4-\sum_{k=1}^{i-1} \frac{3r}{2\pi^2k^2} \big)
\Big)-\frac{3r}{4\pi^2j^2} \\
& \geq &\sum_{k=i}^{j}\frac{3r}{2\pi^2k^2}-\frac{3r}{4\pi^2j^2}\geq \frac{3r}{4\pi^2i^2},
\end{eqnarray*}
we have by  Lemma \ref{L:3.4},
$$
\P_x\big(\wh X^B_{{ \wh \tau^B_{x,j} } }\in W_{i-1}\setminus (V_i\cup W_i)\big)\leq c_3\frac{\E_x[{ \wh \tau^B_{x,j} } ]}{(\frac{3r}{4\pi^2i^2})^\gamma}
\leq c_4i^{2\gamma}\frac{\E_x[{ \wh \tau^B_{x,j} } ]}{r^\gamma}.
$$
Hence, by    the strong Markov property of $X$ and the definition of $W_i$,
\begin{eqnarray}\label{eqn35}
&& \P_x\big(\wh X^B_{{ \wh \tau^B_{x,j} } }\in W_{i-1}\setminus (V_i\cup W_i),\wh \tau^B_U>\wh \tau^B_{B_U(\xi,r)}\big) \nonumber \\
&=& \E_x \Big[  \P_{\wh X^B_{\wh \tau^B_{x, j}}} ( \wh \tau^B_U>\wh \tau^B_{B_U(\xi,r)}); \, \wh X^B_{{ \wh \tau^B_{x,j} } }\in W_{i-1}\setminus (V_i\cup W_i) \big]
\nonumber \\
&\leq & c_42^{1-i}i^{2\gamma}\frac{\E_x[{ \wh \tau^B_{x,j} } ]}{r^\gamma}.
\end{eqnarray}

\noindent(iv) By Lemma \ref{L:3.1'} and the definition of $U_j$, there are constants $c_5,c_6
\geq 1 $  so that for $x\in U_j$,
\begin{equation}\label{eqn36}\begin{aligned}
\begin{split}
&\quad\,\P_x\big(\wh X^B_{{ \wh \tau^B_{x,j} } }\in U \big)=\P_x\big(\wh \tau^B_U>\wh \tau^B_{B_U(x,\frac{3}{4\pi^2j^2}r)}\big)
\leq c_5\sqrt{\E_x[\wh \tau^B_{B_U(\xi,r)}]/(\tfrac{3}{4\pi^2j^2}r)^\gamma}
 \leq c_62^{-j/2}j^\gamma.
\end{split}
\end{aligned}\end{equation}
 Hence by the strong Markov property of $X$   and the definition of $U_j$, for $x\in U_j,$
\begin{eqnarray}\label{eqn37}
 && \P_x\big(\wh X^B_{{ \wh \tau^B_{x,j} } }\in W_j,\wh \tau^B_U>\wh \tau^B_{B_U(\xi,r)}\big)  \nonumber \\
&=&\E_x\big[\P_{\wh X^B_{{ \wh \tau^B_{x,j} } }}(\wh \tau^B_U > \wh \tau^B_{B_U(\xi, r)}) ; \wh X^B_{{ \wh \tau^B_{x,j} } } \in W_j\big] \nonumber \\
&\leq& \P_x\big(\wh X^B_{{ \wh \tau^B_{x,j} } }\in W_j\big)\cdot 2^{-j}   \nonumber \\
&\leq & \P_x\big(\wh X^B_{{ \wh \tau^B_{x,j} } }\in U \big) \cdot  2 \Big(\P_x(\wh \tau^B_U>\wh \tau^B_{B_U(\xi,r)})+\frac{\E_x[\wh \tau^B_{B_U(\xi,r)}]}{r^\gamma}
\Big) \nonumber \\
&\leq&  2c_62^{-j/2}j^\gamma\Big( \P_x(\wh \tau^B_U>\wh \tau^B_{B_U(\xi,r)})+\frac{\E_x[\wh \tau^B_{B_U(\xi,r)}]}{r^\gamma} \Big),
 \end{eqnarray}
where we used  \eqref{eqn36} in the last inequality.

\medskip

Recall that $V_1\subset V_2\subset\cdots\subset V_j$ and $V_1\cup W_1=B_U(\xi, {3r}/4)$, we have
$$\begin{aligned}
V_j\cup\big(\bigcup_{i=2}^{j}W_{i-1}\setminus (W_i\cup V_i)\big)\cup W_j
=V_j\cup \big(\bigcup_{i=2}^{j}(W_{i-1}\setminus W_i)\big)\cup W_j=V_j\cup W_1=B_U(\xi, 3r/4 ).
\end{aligned}$$
Thus it follows from \eqref{eqn33},  \eqref{eqn34}, \eqref{eqn35} and  \eqref{eqn37}  that for $x\in U_j,$
$$\begin{aligned}
&\quad\,\P_x(\wh \tau^B_U>\wh \tau^B_{B_U(\xi,r)})\\
&=\P_x(\wh X^B_{{ \wh \tau^B_{x,j} } }\in V_j,\wh \tau^B_U>\wh \tau^B_{B_U(\xi,r)})+\sum_{i=2}^j\P_x\big(\wh X^B_{{ \wh \tau^B_{x,j} } }\in W_{i-1}\setminus (W_i\cup V_i),\wh \tau^B_U>\wh \tau^B_{B_U(\xi,r)}\big)\\
		&\quad\qquad +\P_x(\wh X^B_{{ \wh \tau^B_{x,j} } }\in W_j,\wh \tau^B_U>\wh \tau^B_{B_U(\xi,r)})+\P_x\big(\wh X^B_{{ \wh \tau^B_{x,j} } }\in U\setminus B_U(\xi,  3r/4),\wh \tau^B_U>\wh \tau^B_{B_U(\xi,r)}\big)\\
		&\leq \lambda_j^{-1}\frac{\E_x[\wh \tau^B_{B_U(\xi,r)}-{ \wh \tau^B_{x,j} } ]}{r^\gamma}+\sum_{i=2}^j c_42^{1-i}i^{2\gamma}\frac{\E_x[{ \wh \tau^B_{x,j} } ]}{r^\gamma}+c_2\frac{\E_x[{ \wh \tau^B_{x,j} } ]}{r^\gamma}\\
		&\quad  +2c_62^{-j/2}j^\gamma\Big(\P_x(\wh \tau^B_U>\wh \tau^B_{B_U(\xi,r)})+\frac{\E_x[\wh \tau^B_{B_U(\xi,r)}]}{r^\gamma}\Big)\\
		&\leq 2\max\{\lambda_j^{-1},c_7\}\frac{\E_x[\wh \tau^B_{B_U(\xi,r)}]}{r^\gamma}+2c_62^{-j/2}j^\gamma
		\Big(\P_x(\wh \tau^B_U>\wh \tau^B_{B_U(\xi,r)})+\frac{\E_x[\wh \tau^B_{B_U(\xi,r)}]}{r^\gamma}\Big)\\
		&\leq (2+2c_62^{-j/2}j^\gamma)\max\{\lambda_j^{-1},c_7,1\}\frac{\E_x[\wh \tau^B_{B_U(\xi,r)}]}{r^\gamma}+2c_62^{-j/2}j^\gamma\P_x(\wh \tau^B_U>\wh \tau^B_{B_U(\xi,r)}),
\end{aligned}$$
where  $c_7:=c_2+\sum_{i=2}^{\infty}c_42^{1-i}i^{2\gamma}$. Since the above estimate works for any $x\in U_j$, we have by the definition of $\lambda_{j+1}$,
\begin{equation}\label{eqn38}
\lambda_{j+1}\geq \frac{1-2c_62^{-j/2}}{2+2c_62^{-j/2}}\min\{\lambda_j,c_8\},
\end{equation}
where $c_8=\min\{1,c_7^{-1}\}$. Now,  fix $j_0$ such that  $2c_62^{-j/2}j^\gamma<1/2$ for every $j\geq j_0$.
By Corollary \ref{C:1} and the fact that $\P_x(\wh \tau^B_U>\wh \tau^B_{B_U(\xi,r)})+\E_x[\wh \tau^B_{B_U(\xi,r)}]r^{-\gamma} \geq 2^{-j}$ for each $x\in V_j$, there is a constant $c_9>0$  so that
$$
\lambda_j\geq c_92^{-j}\ \hbox{ for every }j\leq j_0.
$$
  For $j>j_0$, by using \eqref{eqn38}, we get
$$
\lambda_j\geq \min\{c_92^{-j_0},c_8\}\cdot\prod_{k=j_0}^{j-1}\frac{1-2c_62^{-j/2}}{2+2c_62^{-j/2}}.
$$
This proves $\lambda_j\geq \min\{c_92^{-j_0},c_8\}\cdot\prod_{k=j_0}^\infty \frac{1-2c_62^{-j/2}j^2}{2+2c_62^{-j/2}j^2} > 0 $ for every $j\geq 1$. The proposition follows immediately.
\qed

\smallskip

\begin{proof}[Proof of Proposition \ref{P1}.]
 As mentioned earlier, by   \cite[Proposition 3.1]{CC},
  ${\rm \bf (Jt)}_{\leq, R, \gamma}$  and  ${\bf (EP)} _{\leq, R, \gamma}$  for some $\gamma>0$
  yields  ${\bf (ED)}^s _{\leq, R}$ holds.
So it suffices to show  if for each $R\in (0, R_0),$ {\rm ${\bf (\wh{EP})} _{\leq , R, \gamma}$} and
${\bf (\wh{Jt})}_{\leq, R, \gamma}$ hold with some $\gamma>0,$
then ${\bf (\wh{ED})}^s _{\leq, R}$  holds.

Let $x_0\in\sX$ and $B=B(x_0, R)$ be a ball
  so that $\overline{B} \subsetneq \sX$.
Let $U\subset 2^{-1}B$ be an open subset.
Note that for each  $x\in U\cap A(x_0, 3R/8, R/2),$ we have $B_U(x, R/16)\subset U\cap A(x_0, R/4, R/2).$
Thus  for each  $x\in U\cap A(x_0, 3R/8, R/2),$
$$ \P_x  \big(\wh X^B_{\wh\tau^B_{U\cap A(x_0, R/4, R/2)}} \in U \big)\leq \P_x (\wh X^B_{\wh\tau^B_{B_U (x, R/16)}} \in U).$$
Suppose  {\rm ${\bf (\wh{EP})} _{\leq , R, \gamma}$} and
${\bf (\wh{Jt})}_{\leq, R, \gamma}$ hold  with some  $\gamma>0$.
Then by applying Proposition \ref{P:3.5}, there is  a positive constant $C=C(R, \gamma)$ independent of
 the center $x_0\in\sX$ of
 the ball $B(x_0, R)$ such that for each  $x\in U\cap A(x_0, 3R/8, R/2),$
$$
\P_x (\wh X^B_{\wh\tau^B_{B_U (x, R/16)}} \in U)\leq   C\dfrac{\E_x[\wh \tau^B_{B_U(x, R/16)}]}{R^\gamma}\leq C\dfrac{\E_x[\wh \tau^B_{U\cap A(x_0, R/4, R/2)}]}{R^\gamma}.
$$
Combining the two displays above,  ${\bf (\wh{ED})}^s _{\leq, R}$   holds.
Thus the proof is complete.
 \end{proof}

\section{Examples}\label{S:4}

In this section,  we present some examples of Theorem \ref{T0}.
We show that the BHP  holds for  a large class of symmetric diffusion with jumps associated with the regular symmetric Dirichlet forms on
 metric measure spaces having volume doubling and reverse volume doubling properties  that admits a two-sided  heat kernel estimates of the mixture of (sub-)Gaussian and stable-like form, a class of non-symmetric diffusion processes with jumps in non-divergence form in $\R^d$ and an example of a  L\'evy process with degenerate Gaussian component.

\subsection{Symmetric discontinuous  processes with diffusive part
 on  metric measure spaces}

Let $ ( \sX, d, m)$ be a metric measure space and $m$  a positive
Radon measure on $\sX$ with full support. We assume that all balls are
relatively compact and assume for simplicity that $m (\sX)=\infty.$

For each $x\in \sX$ and $r>0,$
let  $V(x, r):= m(B(x, r))$.
We say  the metric measure space $(\sX, d, m)$
is volume doubling (VD) if there is $C_1\in (1,\infty)$ such that
$$
V(x,2r) \leq C_1 \, V(x,r) \qquad \hbox{ for every }  x\in \sX   \hbox{ and }   r\in (0,\infty).
$$
This is equivalent to the existence of positive constants $c_1$ and $d_1$ so that
\begin{equation}\label{e:vd}
\frac{ V(x, R)}{V(x, r)} \leq c_1  \Big(\frac{R}{r} \Big)^{d_1}  \quad
\hbox{ for every }x\in \sX   \hbox{ and }   0<r\leq R<\infty.
\end{equation}
We say that reverse volume doubling property (RVD) holds if there are positive constants $c_2$ and $d_2$ so that
\begin{equation}\label{e:rvd}
\frac{V(x,R)} {V(x,r)} \geq  c_2 \Big(\frac{R}{r} \Big)^{d_2}  \quad
\hbox{ for every }x\in \sX   \hbox{ and }   0<r \leq R  \leq {\rm diam}(\sX).
\end{equation}
This is equivalent to the existence of $\lambda_0\geq 2 $ and $C_2 >1$ so that
\begin{equation}\label{e:rvd2}
V(x, \lambda_0 r)   \geq C_2  V(x, r) \quad \hbox{ for every }x\in \sX   \hbox{ and }   0<r    \leq   {\rm diam}(\sX)/\lambda_0.
\end{equation}
If $\sX$ is connected and unbounded, (VD) implies (RVD) ; See \cite[Proposition 2.1 and  the  paragraph before Remark 2.1]{GH}.

We consider a regular symmetric Dirichlet form $(\mathcal E, \mathcal F)$  having both the strongly local term and  the pure-jump term,
and having no killing term.
That is,
\begin{equation}\label{e:E}
\begin{split}
\mathcal E(f,g)=& \, \mathcal E^{(c)}(f,g)+\int_{\sX\times \sX\setminus \textrm{diag}}(f(x)-f(y)(g(x)-g(y))\,J(dx,dy)\\
=&:\mathcal E^{(c)}(f,g)+\mathcal E^{(j)}(f,g),\qquad f,g\in \mathcal F,\end{split}
\end{equation}
where $(\mathcal E^{(c)},\mathcal F)$ is the strongly local part of  $(\mathcal E, \mathcal F)$
(namely $\mathcal E^{(c)}(f, g)=0$ for all $f, g\in \mathcal F$ having $(f-c)g=0$ $\mu$-a.e.\ on $\sX$ for some constant $c\in \R$)
 and  $J(\cdot,\cdot)$ is a symmetric Radon measure on $\sX\times \sX\setminus \textrm{diag}.$
 We assume that neither $\mathcal E^{(c)}(\cdot,\cdot)$ nor $J(\cdot,\cdot)$ are identically zero.

Let $\R_+:=[0,\infty)$ and  $\phi_c: \R_+\to \R_+$ (resp. $\phi_j: \R_+\to \R_+$)  be a strictly increasing continuous
function  with $\phi_c (0)=0$ (resp. $\phi_j(0)=0$),
$\phi_c(1)=1$ (resp. $\phi_j(1)=1$)
and satisfying that there exist constants $c_{1,\phi_c},c_{2,\phi_c}>0$ and $\beta_{2,\phi_c}\ge \beta_{1,\phi_c}>1$ (resp. $c_{1,\phi_j},c_{2,\phi_j}>0$ and $\beta_{2,\phi_j}\ge \beta_{1,\phi_j}>0$)
such that
\begin{equation}\label{polycon}
\begin{split}
 &c_{1,\phi_c} \Big(\frac Rr\Big)^{\beta_{1,\phi_c}} \leq
\frac{\phi_c (R)}{\phi_c (r)}  \ \leq \ c_{2,\phi_c} \Big(\frac
Rr\Big)^{\beta_{2,\phi_c}}
\quad \hbox{for all }
0<r \le R. \\
\bigg(\hbox{resp. }  &c_{1,\phi_j} \Big(\frac Rr\Big)^{\beta_{1,\phi_j}} \leq
\frac{\phi_j (R)}{\phi_j (r)}  \ \leq \ c_{2,\phi_j} \Big(\frac
Rr\Big)^{\beta_{2,\phi_j}}
\quad \hbox{for all }
0<r \le R.\bigg)\end{split}\end{equation}

We assume that
 \begin{equation}\label{e:1.11}
 \phi_c(r)\le \phi_j(r)   \hbox{ for } r\in (0,1] \quad \hbox{ and } \quad \phi_c(r)\ge \phi_j(r) \hbox{  for } r\in [1,\infty).
 \end{equation}
 Since $\beta_{1,\phi_c}>1$, by  \cite[Definition, p.\ 65; Definition, p.\ 66; Theorem 2.2.4 and its remark, p.\ 73]{BGT}, there exists a strictly increasing
 continuous  function $\bar \phi_c (r): \R_+\to \R_+$
  such that there are  constants $c_2\geq c_1>0$ so that
 \begin{equation}\label{e:1.12}
  c_1 \frac{ \phi_c(r)}{r} \le \bar \phi_c (r)\le c_2 \frac{\phi_c(r)}{r}  \quad \hbox{for all } r>0.
  \end{equation}

Define
\begin{equation}
p^{(c)}(t,x,y):=
\frac{1}{V(x,\phi_c^{-1}(t))} \exp\left(- \frac{d(x,y)}{\bar \phi_c^{-1}(t/d(x,y))}\right)
\end{equation}
and
\begin{equation}
p^{(j)}(t, x, y):= \frac{1}{V(x, \phi_j^{-1}(t))}\wedge \frac{t}{V(x, d(x, y)) \phi_j (d(x, y))}.
\end{equation}

\begin{defn} {\bf ($\bf {HK}(\phi_c, \phi_j)$ condition)} \rm
 We say that
${\rm HK}(\phi_c, \phi_j)$
holds if there exists a heat kernel $p(t, x,y)$
of the semigroup $\{P_t\}_{t\ge0}$ associated with $(\mathcal E,\mathcal F)$ in \eqref{e:E}
and the following estimates hold
for all $t>0$ and all $x,y\in \sX$,
\begin{equation}\label{HK}\begin{split}
&   c_1\Big(\frac 1{V(x,\phi_c^{-1}(t))}\wedge \frac 1{V(x,\phi_j^{-1}(t))} \wedge
 \big(p^{(c)}(c_2 t,x,y)+p^{(j)}(t,x,y)\big)\Big) \\
 & \le \   p(t, x,y)  \\
  &\le c_3\Big(\frac 1{V(x,\phi_c^{-1}(t))}\wedge \frac 1{V(x,\phi_j^{-1}(t))} \wedge
\big(p^{(c)}(c_4 t,x,y)+p^{(j)}(t,x,y)\big)\Big),
\end{split}
\end{equation}
where $c_k>0$, $k=1, \cdots, 4$,  are constants independent of $x,y\in \sX$ and $t>0$.

\end{defn}

\begin{prp}\label{C2}
Assume that the metric measure space $(\sX, d , m)$ satisfies
${\rm (VD)}$ and ${\rm (RVD)}$,
the scale functions $\phi_c$ and $\phi_j$ satisfy  \eqref{polycon} and \eqref{e:1.11}.
Suppose $X$ is the $m$-symmetric Hunt process associated with the regular Dirichlet form $(\mathcal E,\mathcal F)$ in \eqref{e:E} and satisfies
the condition ${\rm HK(\phi_c, \phi_j)}$. Let $D$ be an  open subset in $\sX$. Then  for each $z_0\in \partial D$ and $R>0$
so that $\overline{B(z_0, 2R)}\subsetneq \sX$,
there exists a constant  $C=C(z_0, R)\geq 1$  independent of the geometry of $D$  such that   for   any  nonnegative  regular harmonic functions $f$ and $g$
on $D\cap B(z_0, R)$   vanishing  on  $(D^c)^r\cap B(z_0, R),$
\begin{equation}\label{e:4.11}
f(x)g(y)\leq Cf(y)g(x)  \quad \hbox{for } x, y\in D\cap B(z_0, R/32).
\end{equation}
Furthermore, if in addition we assume there exist a positive constant $c$ and a positive function $V_1$ on $\R_{+}$ such that
 $V(x, r)\geq cV_1(r)$ for any $x\in \sX$ and $r>0,$  then  BHP  \eqref{e:4.11} holds
   with $C=C(R)\geq 1$.

\end{prp}

\smallskip

\begin{exa}\label{E:4.3}{\bf (Symmetric diffusion with jumps on $\R^d$.)} \quad \rm

A prototype of  Proposition \ref{C2} is the Dirichlet form $(\mathcal{E}, \D(\mathcal{E}))$ in $L^2(\R^d)$ which is given by
\begin{equation}\label{e:E1}\begin{aligned}
\mathcal{E}(u, v)=\dfrac{1}{2}&\int_{\R^d} A(x)\nabla u(x)\cdot \nabla v(x) \,dx\\
&+\int_{\R^d\times\R^d} (u(x)-u(y))(v(x)-v(y)) J(x, y)\,dx\,dy
\end{aligned}\end{equation}
and $\D(\mathcal{E})=\F=\overline{C_c^1(\R^d)}^{\mathcal{E}_1}$ ,
where  $\mathcal{E}_1(u, v):=\mathcal{E}(u, v)+\int_{\R^d} u(x)v(x)\,dx.$
Here $(a_{ij}(x))_{1\leq  i, j\leq d}$ is a symmetric $d\times d$ matrix-valued function on $\R^d$
that is   uniformly bounded and elliptic.
The jump density function $J(x, y)$ satisfies that there are positive constants $c_k, k=1, 2$ so that
$$
 \frac{c_1 }{|x-y|^d\phi_0  (|x-y|)}\leq
J(x, y)  \leq \frac{c_2 }{|x-y|^d\phi_0  (|x-y|)}
\quad \hbox{for   } (x, y) \in \R^d\times \R^d \setminus {\rm diag},
$$
where $\phi_0$ is a  strictly increasing function on $[0, \infty)$
such that
$\phi_0 (0)=0$, $\phi_0 (1)=1$  and
there exist constants $c \ge1$ and
 $0<\aa_* \leq \aa^*<2$
so that
$$
c^{-1}  \Big(\frac Rr\Big)^{\aa_*} \le
\frac{\phi_0 (R)}{\phi_0 (r)} \le c\Big(\frac
Rr\Big)^{\aa^*}
\quad \hbox{for every } 0<r<R<\infty.
$$
It follows from \cite{CK3} that   associated with $(\mathcal{E}, \mathcal F)$ in \eqref{e:E1} is a symmetric strong Markov process $X$ in $\R^d$ satisfying
${\rm HK}(\phi_c, \phi_j)$ with $\phi_c(r)=r^2$ and $\phi_j(r)=\phi_0(r)$.
 Hence by Proposition \ref{C2},   for each $R>0,$ the BHP \eqref{e:4.11}  holds
 with $C=C(R)\geq 1$ uniformly for $z_0\in \partial D.$

\end{exa}

\begin{exa}\label{E:4.4}{\bf ( Symmetric diffusion with jumps on $d$-set with walk dimension $\beta \geq 2$.)} \quad \rm
Let  $(\sX, d , m)$ be an Alfhors $d$-regular set.
Suppose that there is a $m$-symmetric diffusion $\{Z_t, t\ge0\}$ on $\sX$
such that
it has a transition density function $q(t, x, y)$  with respect to the measure $m$ that has the following two-sided estimates:
\begin{equation}\label{e:1.7}
q(t, x, y) \asymp t^{-d/\beta} \exp \left( - \left(\frac{  d(x, y)^\beta }{ t} \right)^{1/(\beta -1)}  \right),
\quad t>0, x, y\in \sX
\end{equation}
 for some $\beta \geq 2$. Denote by $(\bar{\mathcal E},\bar{\mathcal F})$ the corresponding Dirichlet form.
   A prototype
  is a Brownian motion on the  $n$-dimensional   unbounded Sierpi\'{n}ski gasket; see, for instance, \cite{BP}.
  In this case, $d=\log (n+1)/\log 2$ is the Hausdorff dimension of
the gasket, and $  \beta=\log (n+3)/\log 2$ is called the walk dimension in \eqref{e:1.7}.

Take any $\alpha \in (0, \beta)$ and a symmetric strongly local regular Dirichlet form $(\mathcal E^{(c)}, \bar{\mathcal F})$
in $L^2(\sX; m)$ with the property that $\mathcal E^{(c)} (f,f)\asymp \bar{\mathcal E} (f,f)$ for all
$f\in \bar{\mathcal F}$. Consider the following regular Dirichlet form  $(\mathcal E,\bar{\mathcal F})$ in $L^2(\sX; m)$ defined by
 \begin{equation}\label{eq:DFdsetw}
\mathcal E (u,v)=
\mathcal E^{(c)} (u,v)
+ \int_{\sX}\int_{\sX} (u(x)-u(y))(v(x)-v(y))
\frac{c(x,y)}{d(x,y)^{d+\alpha}}\,m(dx)\,m(dy),
\end{equation}
where
$c(\cdot,\cdot)$ is a symmetric measurable
function on $\sX\times \sX$
that is bounded between two positive constants.
It is shown in \cite[Example 7.2]{CKW3} that $(\mathcal E,\bar{\mathcal F})$ enjoys
 ${\rm HK} (\phi_c, \phi_j)$ with  $\phi_c (r)= r^\beta$ and $\phi_j (r)=r^\alpha.$
 By the definition of  Alfhors $d$-regular set, there exits $c>0$ such that $V(x, r)=m(B(x, r))\geq cr^d$ for any $x\in \sX$ and $r>0.$
 Hence by  Proposition \ref{C2},  for each $R>0,$ the BHP \eqref{e:4.11}  holds
 with $C=C(R)\geq 1$ uniformly for $z_0\in \partial D$.
\end{exa}

\smallskip

In the following, we give the proof of  Proposition \ref{C2}.
Given $\phi_c$ and $\phi_j$ satisfying \eqref{e:1.11},  set
\begin{equation}\label{e:scaling-}
\phi(r):=\phi_c(r)\wedge \phi_j(r)
=\begin{cases} \phi_c(r),\quad& r\in (0, 1], \\
\phi_j(r),\quad &r\in [1, \infty).
\end{cases}
\end{equation}
It is clear that $\phi(r)$ is a strictly increasing function on $\R_+$ with $\phi(0)=0$ and $\phi(1)=1$, and satisfies that there
exist constants $c_{1,\phi},c_{2,\phi}>0$ so that
\begin{equation}\label{e:4.2}c_{1,\phi} \Big(\frac Rr\Big)^{\beta_{1,\phi}} \leq
\frac{\phi (R)}{\phi (r)}  \ \leq \ c_{2,\phi} \Big(\frac
Rr\Big)^{\beta_{2,\phi}}
\quad \hbox{for all }
0<r \le R,\end{equation}  where $\beta_{1,\phi}=\beta_{1,\phi_c}\wedge \beta_{1,\phi_j}$ and $\beta_{2,\phi}=\beta_{2,\phi_c}\vee \beta_{2,\phi_j}$.

Suppose $X$ is the $m$-symmetric Hunt process associated with the regular Dirichlet form $(\mathcal E,\mathcal F)$ in \eqref{e:E} and satisfies
the condition ${\rm HK}(\phi_c, \phi_j)$.
By Proposition 3.2 (i) in \cite{CKW3},  $(\mathcal E, \mathcal F)$ is conservative.
For each open set $B,$ denote by $\tau_B$ the first exit time of $X$ from $B.$
The following  Lemma is  due to Theorems 1.13 and 1.14 in \cite{CKW3}.

\begin{lem}\label{L:4.1}
There exists  $C>1$ such that for any $x\in \sX$ and  $r>0,$
$$
 C^{-1}\phi(r)\leq \E_x\tau_{B(x, r)}  \leq C\phi(r).
$$
Hence assumption (A5) holds.
\end{lem}

For each ball $B$ in $\sX,$ denote by $G_B(x, y)$ the Green function of $X$ in $B$ with respect to the reference measure $m.$

\begin{lem}\label{L:4.2}
For each $x_0\in \sX$ and any  $r>0,$
\begin{equation}\label{e:4.4}
\sup_{x\in B(x_0, r/16)}\sup_{y\in B(x_0, r/2)\setminus B(x_0, r/8)}G_{B(x_0, r)}(x, y)<\infty.
\end{equation}
Hence assumption (A3) holds for $X$ in $\sX$.
 Furthermore,  if in addition we assume there exist a positive constant $c$ and a strictly positive function $V_1$ on $\R_{+}$ such that
 $V(x, r)\geq cV_1(r)$ for any $x\in \sX$ and $r>0,$ then assumption (A3') holds for $X$ in $\sX$.
\end{lem}

\pf  It follows from \cite[Theorem 1.18]{CKW3} that the uniform elliptic Harnack inequality holds for $X.$
Let $\ee_0\in (0, 1/8)$ which will be determined later.
Note that for each $x\in B(x_0, r),$ $G_{B(x_0, r)}(x, \cdot)$ is harmonic in $B(x_0, r)\setminus \{x\}.$ By the uniform elliptic Harnack inequality and  standard chain argument, there exists $c_1>0$ such that for any $x_0\in \sX, r>0$ and $x\in B(x_0, \ee_0 r/2),$
$$\sup_{y\in B(x_0, r/2) \setminus B(x_0, \ee_0 r)}G_{B(x_0, r)}(x, y)\leq c_1\inf_{u\in B(x_0, r/2) \setminus B(x_0, \ee_0 r)}G_{B(x_0, r)}(x, u).$$
Hence by this inequality and Lemma \ref{L:4.1}, for $x\in B(x_0, \ee_0 r/2)$ and $y\in B(x_0, r/2) \setminus B(x_0, \ee_0 r),$
\begin{equation}\label{e:4.14}\begin{aligned}
G_{B(x_0, r)}(x, y)&\leq c_1\dfrac{1}{m(B(x_0, r/2)\setminus B(x_0, \ee_0 r))} \int_{B(x_0, r/2) \setminus B(x_0, \ee_0 r)} G_{B(x_0, r)}(x, u)\,m(du)\\
&\leq c_1\dfrac{1}{V(x_0, r/2)-V(x_0, \ee_0 r)}\E_x \left[ \tau_{B(x_0, r)} \right]\\
&\leq c_2\dfrac{1}{V(x_0, r/2)-V(x_0, \ee_0 r)}\phi(r).
\end{aligned}\end{equation}
By the (RVD) condition \eqref{e:rvd2}, we can take $\ee_0$ small enough such that there exists $c_3>1$ independent of $x_0\in \sX$ so that
$$
V(x_0, r/2)-V(x_0, \ee_0 r)\geq (c_3-1)V(x_0, \ee_0 r).
$$
Hence by this together with \eqref{e:4.14},  we have
\begin{equation}\label{e:4.16}
\begin{aligned}
&\sup_{x\in B(x_0, \ee_0 r/2), y\in B(x_0, r/2) \setminus B(x_0, r/8)}G_{B(x_0, r)}(x, y)\\
&\leq \sup_{x\in B(x_0, \ee_0 r/2), y\in B(x_0, r/2) \setminus B(x_0, \ee_0r)}G_{B(x_0, r)}(x, y)\leq \dfrac{c_2}{c_3-1}\dfrac{\phi(r)}{V(x_0, \ee_0 r)}<\infty.
\end{aligned}\end{equation}
Note that for each $y\in B(x_0, r/2)\setminus B(x_0, r/8), G_{B(x_0, r)}(\cdot, y)$ is harmonic in $B(x_0, r/8).$ Then for each $x\in B(x_0, r/16)\setminus B(x_0, \ee_0r/2)$ and $y\in B(x_0, r/2)\setminus B(x_0, r/8),$  one can use  the standard chain argument, \eqref{e:4.16} and the uniform elliptic Harnack inequality to obtain
 \eqref{e:4.4}.

 Furthermore,  if in addition we assume there exist a positive constant $c$ and a positive function $V_1$ on $\R_{+}$ such that
 $V(x, r)\geq cV_1(r)$ for any $x\in \sX$ and $r>0,$ then \eqref{e:4.16} holds uniformly for $x_0\in \sX.$
This together with the standard chain argument and  uniform elliptic Harnack inequality   yields that \eqref{e:4.4} holds uniformly for $x_0\in\sX.$
That is,  assumption (A3') holds.
\qed

\smallskip

The following result follows from  \cite[Lemma 3.4]{CKW3} and \eqref{e:4.2}.

\begin{lem}\label{L:4.3}
For each $R>0,$ there are constants  $C_k>0, k=1, 2$  such that for any $x\in \sX, r\in (0, R)$ and $t>0,$
$$
\P_x(\tau_{B(x,r)}\leq t)\leq  C_1  \frac{ t}{\phi(r)}\leq C_2 \dfrac{t}{r^{\beta_{2, \phi}}}.
$$
Hence ${\bf (EP)} _{\leq, R, \gamma}$ holds for $X$ for  $\gamma=\beta_{2, \phi}.$
\end{lem}

\begin{lem}\label{L:4.4}
There is a symmetric and positive function $J(x, y)$ in $\sX\times \sX$ such that  $J(dx, dy)=J(x, y) \,m(dx)\, m(dy)$ in $\sX\times \sX.$
Moreover, there are positive constants $C_k, k=1, 2$ so that
\begin{equation}\label{e:4.15}
 \frac{C_1 }{V(x,  d(x, y))\phi_j  (d(x, y) )}\leq
J(x, y)  \leq \frac{C_2 }{V(x,  d(x, y))\phi_j  (d(x, y) )}
\quad \hbox{for   } (x, y) \in \sX\times \sX \setminus {\rm diag}.
\end{equation}
Furthermore, for each $R>0,$ there is a constant $C_3>0$ such that for every $x\in \sX$ and $r\in (0, R)$,
\begin{equation}\label{e:4.6}
\int_{B(x, r)^c} J(x, y) m(dy)\leq \frac{C_3}{r^{\beta_{2, \phi}}}.
\end{equation}
Hence assumption (A4') and ${\rm \bf (Jt)}_{\leq, R, \beta_{2, \phi}}$ hold for $X$.
\end{lem}

\pf By Theorems 1.13 and 1.14 in \cite{CKW3}, \eqref{e:4.15} holds.
By  using \eqref{e:4.15}, \eqref{polycon} and  (VD) and (RVD) condition,   assumption (A4') holds.

By  \cite[Lemma 2.1]{CKW1},  there are constants $c_k, k=1, 2>0$ such that for every $x\in \sX$ and $r>0$,
\begin{equation}\label{e:4.7}
\int_{B(x, r)^c} J(x, y) m(dy)\leq c_1\int_{B(x, r)^c} \frac1{V(x, d(x,y)) \phi_j (d(x, y))} m(dy) \leq \frac{c_2}{\phi_j(r)}.
\end{equation}
 Hence \eqref{e:4.6} holds by \eqref{e:4.7} and \eqref{polycon}.
 Thus ${\rm \bf (Jt)}_{\leq, R, \beta_{2, \phi}}$ holds for $X$.
\qed

\medskip

\begin{proof}[Proof of  Proposition \ref{C2}.]

Since $X$ is the $m$-symmetric Hunt process associated with the regular Dirichlet form $(\mathcal E,\mathcal F)$ in \eqref{e:E} and satisfies
the condition ${\rm HK(\phi_c, \phi_j)}$, by the argument below assumption (A1), for each ball $B$ in $\sX,$ the subprocess $X^{B}$ is symmetric with respect to the measure $m$ restricted in $B,$ and for each $x\in \sX$ and $\alpha>0,$ the $\alpha$-resolvent measure  $G^{\alpha, B}(x, \cdot)$ of the subprocess $X^{B}$  is absolutely continuous with respect to $m(dy)$ in $B.$
As the reference measure $m(dx)$  is  an excessive measure of $X$,
it follows from \cite{Si} that every semi-polar set of $X$ is polar.  Hence  assumption (A1) holds.
By Lemmas \ref{L:4.2} and \ref{L:4.4},  assumptions (A3) and (A4') hold.
Moreover, by Lemmas  \ref{L:4.3} and \ref{L:4.4}, for each $R>0,$ the conditions ${\bf (EP)} _{\leq, R, \beta_{2, \phi}}$  and ${\rm \bf (Jt)}_{\leq, R, \beta_{2, \phi}}$ hold for $X.$
Thus assumption (A2') holds by Remark \ref{R:1.11} and Proposition \ref{P1}.
Hence assumptions (A1), (A2'), (A3) and (A4') hold. Thus the BHP \eqref{e:4.11} of  Proposition \ref{C2} holds by Theorem \ref{T0}.

Furthermore, if in addition we assume there exist a positive constant $c$ and a positive function $V_1$ on $\R_{+}$ such that
 $V(x, r)\geq cV_1(r)$ for any $x\in \sX$ and $r>0,$  then by Lemma \ref{L:4.2},  assumption (A3') holds.
Note that  assumption (A5) holds by Lemma \ref{L:4.1}.
By combining the argument above, assumptions (A1), (A2')-(A4') and (A5) hold.
Hence in this case, by Theorem \ref{T0},  BHP \eqref{e:4.11} holds
with $C=C(R)\geq 1$.
  \end{proof}

\subsection{Non-symmetric diffusion with jumps in non-divergence form in $\R^d$}

Consider the following type of non-local operators on $\R^d (d\geq 3)$:
 \begin{equation}\label{e:L}
\L^b f(x):=\L^0 f(x)+b(x)\cdot \nabla f(x) + \S f(x),
\end{equation}
where $\L^0$ is an elliptic operator of second order in non-divergence form,
that is,
\begin{equation}\label{e:nd}
\L^0 f(x):=\sum_{i, j=1}^d a_{ij}(x)\dfrac{\partial^2}{\partial x_i\partial x_j}f(x),
\end{equation}
and
\begin{equation}\label{e:1.4}
\S f(x):=
\int_{\R^d} \left( f(x+z)-f(x)- \nabla f(x) \cdot
z \1_{\{|z|\leq 1\}}  \right) j(x, z) \,dz, \quad f\in C_b^2(\R^d).
\end{equation}

We assume  $(a_{ij}(x))_{1\leq  i, j\leq d}$ is a symmetric $d\times d$ matrix-valued function on $\R^d$
that is   uniformly bounded, elliptic and H\"older continuous, the function  $b$ belongs to the Kato class $K_d^1$; that is,
\begin{equation}\label{e:b}
\lim_{r\rightarrow 0}\sup_{x\in\R^d}\int_{B(x, r)} |y-x|^{1-d}|b|(y) dy=0.
\end{equation}
Assume $j(x, z)$ is a nonnegative locally bounded  function on $\R^d\setminus \{0\}$ such  that
$\int_{\R^d}  (1\wedge | z|^2) j(x, z)\,dz<\infty$ and there exist constants $\beta\in (0, 2)$ and  $c_0>0$ so that
\begin{equation}\label{e:j}
j(x, z) \leq \frac{c_0}{ |z|^{d+\beta} }  \quad \hbox{for }  z\in \R^d\setminus \{0\}.
\end{equation}
 By \cite{CHXZ}, there is an irreducible  Hunt process $X^b$ that has a jointly continuous  transition density function $p^b(t,x,y)$ with respect to the Lebesgue measure of $\R^d.$ The Hunt process $X^b$ has a L\'evy system $(J(x, y)\,dy, dt),$ where $J(x, y):=j(x, y-x).$

 \begin{prp}\label{C6}
Let $D$ be an  open subset in $\R^d (d\geq 3)$ and  $R_0>0$.
Assume for each $x\in\R^d,$ $z\mapsto j(x, z)$ is positive and for each  $R\in (0, R_0)$ and $x_0\in\R^d,$
\begin{equation}\label{e:4.27}
\inf_{z\in \R^d:  R/16\leq |z|<R/2} j(x_0, z)>0.
\end{equation}
Suppose  for each  $x_0\in\R^d, R\in (0, R_0)$ and $0<p<q\leq 1,$ there exists $c_1=c_1(x_0, R, p, q)>1$  such that  for any $x\in B(x_0, pR)$ and $y\in B(x_0, qR)^c,$
\begin{equation}\label{e:4J}
c_1^{-1}J(x_0, y)\leq J(x, y) \leq c_1J(x_0, y).
\end{equation}
Then  for each $z_0\in \partial D$ and $R\in (0, \tfrac{1}{16}R_0),$
there exists a constant  $C=C(z_0, R)\geq 1$   independent of the geometry of $D$ such that   for   any  nonnegative  regular harmonic functions $f$ and $g$
on $D\cap B(z_0, R)$   vanishing  on  $(D^c)^r\cap B(z_0, R),$
\begin{equation}\label{e:Xb}
f(x)g(y)\leq Cf(y)g(x)  \quad \hbox{for } x, y\in D\cap B(z_0, R/32).
\end{equation}
Furthermore, if \eqref{e:4.27} and \eqref{e:4J} hold uniformly for $x_0\in\R^d,$
then BHP \eqref{e:Xb} holds with $C=C(R)\geq 1$ uniformly for $z_0\in \partial D$.
\end{prp}

\begin{exa}\label{E:4.9}
Suppose  there are positive constants $c_k, k=1, 2$ so that
$$
 \frac{c_1 }{|z|^d\phi(|z|)}\leq
j(x, z)  \leq \frac{c_2 }{|z|^d\phi(|z|)}
\quad \hbox{for   } (x, z) \in \R^d\times \R^d,
$$
where
$\phi$ is a  strictly increasing function on $[0, \infty)$
such that
$\phi (0)=0$, $\phi (1)=1$  and
there exist constants $c_3 \ge1$ and
 $0<\aa_* \leq \aa^*<2$
so that
$$
c_3^{-1}  \Big(\frac Rr\Big)^{\aa_*} \le
\frac{\phi (R)}{\phi (r)} \le c_3\Big(\frac
Rr\Big)^{\aa^*}
\quad \hbox{for every } 0<r<R<\infty.
$$
It is easy to check that \eqref{e:4.27} and \eqref{e:4J} hold uniformly for $x_0\in\R^d.$
By  Proposition \ref{C6}, the BHP
\eqref{e:Xb} holds with $C=C(R)\geq 1$.

\end{exa}

In the following, we give the proof of   Proposition \ref{C6}.
 By \cite{CHXZ},  there are  positive constants
$c_k $, $k=1, 2, 3$
 so that for $t>0,$
\begin{equation}\label{e:p}
p^b(t,x,y) \leq c_1e^{c_2t} (t^{-d/2}\wedge t^{-d/\beta})\wedge \left(t^{-d/2} e^{-c_3|y-x|^2/t}+t|x-y|^{-(d+\beta)}\right) \quad \hbox{for} \quad
x, y\in \R^d.
\end{equation}

For each ball $B,$ denote by $G^b_B(x, y)$ the Green function of $X^b$ in $B$ with respect to the Lebesgue measure.
By \cite[Lemmas 4.7 and 4.9]{CW1}, we have the following result.

\begin{prp}\label{P:G}
For each $R>0,$ there is a constant
  $C=C(\sL^b, R)>1$    such that for each ball $B$ in $\R^d$  with radius $R>0$,
 $$
\frac{C^{-1}} {|x-y|^{d-2}} \left(1\wedge \frac{\delta_B(x)}{|x-y|}\right) \left(1\wedge \frac{\delta_B(y)}{|x-y|}\right)
 \leq G^b_B (x, y)
 \leq  \frac{C} {|x-y|^{d-2}} \left(1\wedge \frac{\delta_B(x)}{|x-y|}\right) \left(1\wedge \frac{\delta_B(y)}{|x-y|}\right)
$$
for every $ x\not=y\in B$.
\end{prp}

Let $B$ be a ball
in $\R^d$. Denote by $X^{b,B}$ the subprocess of $X^b$ killed upon exiting $B.$
By \eqref{e:p} and an argument similar to that of \cite[Theorem 3.4]{CKS}, $X^{b,B}$ has a
jointly continuous transition density
$p^b_B(t,x,y)$ for $x, y\in B$ with $x\neq y$ with respect to the Lebesgue measure. Using\vadjust{\goodbreak} the continuity of $(x, y)\mapsto p^b_B(t,x,y)$ and the
estimate \eqref{e:p}, we have the following result.

\begin{prp}\label{pHC}
$X^{b,B}$ is a Hunt process, and it satisfies the strong Feller
property, that is, for every $f \in L^{\infty}(B)$,
$P^B_tf(x):=\E_x[f(X^{b,B}_t)]$ is bounded and continuous in $B$.
\end{prp}

Define
\[
h_B(x) := \int_B G^b_B(y,x) \,dy \quad\mbox{and}\quad \xi
_B(dx):=h_B(x) \,dx.
\]

 \begin{prp}\label{em}
Let $B=B(x_0, R)$ be a ball in $\R^d.$
Then $\xi_B$ is an excessive measure for $X^{b,B}$, that is, for every
Borel function $f \ge0$,
$$
\int_B f(x) \xi_B(dx) \ge\int_B \E_x[f(X^{b,B}_{t})]
\xi_B(dx).
$$
Moreover, $h_B$ is a strictly positive, bounded continuous function on $B$ and
there exists  $C=C(d, R)>1$ independent of $x_0\in\R^d$ so that
\begin{equation}\label{e:h}
C^{-1}\leq h_B(x)\leq C \quad \hbox{for }   x\in \tfrac{1}{2}B.
\end{equation}
\end{prp}

\pf By Proposition \ref{P:G} and an argument similar to that of \cite[Proposition 5.2]{CKS}, it suffices to prove \eqref{e:h}.
Let $B=B(x_0, R)$ be a ball in $\R^d.$
 By  Proposition \ref{P:G}  and \cite{GW, Zhao}, the Green function $G^b_B$ is  comparable
 to the Green function  $G^\Delta_B$ of the Laplace operator $\Delta$ in $B.$
Denote by $\tau^W_B$ the first exit time of Brownian motion $W$ from $B.$
Then for each $x\in B,$
$$h_B(x)= \int_B G^b_B(y,x) \,dy\asymp   \int_B G^\Delta_B(y,x) \,dy=  \E_x \tau^W_B,$$
where the multiplicative constants are independent of the center $x_0$ of the ball $B(x_0, R).$
 By Brownian scaling, there is a positive constant $c =c(d)>0$ so that for each $x\in \tfrac{1}{2}B,$
$$
c \, (R/2)^2 =  \E_x \tau^W_{B(x, R/2)} \leq \E_x\tau^W_B\leq \E_x \tau^W_{B(x, 2R)} = c\, (2R)^2 .
$$
Hence  \eqref{e:h} holds.
\qed

\smallskip

We define a transition density with respect to the reference measure
$\xi_B$ by
\[
\overline{p}^{b, B}(t,x,y) := \frac{ p^b_B(t,x,y)}{h_B(y)}.
\]
Let
\[
\overline{G}{}^b_B(x,y) :=
\int^{\infty}_0 \overline{p}^{b, B}(t, x, y)\,dt = \frac{ G^b_B(x,y)}{h_B(y)}.
\]
 Then $\overline{G}{}^b_B(x,y)$ is
the Green function of $X^{b,B}$ with respect to the reference measure~$\xi_B$.
For   each open set $U$ in $\R^d,$ define $\tau^b_U:=\inf\{t>0: X^b_t\notin U\}.$

\begin{defn}\label{dhar}
Let $B$ be a ball in $\R^d.$
Suppose that $U$ is an open subset of $B$. A Borel function $u$ on
$B$ is said to be:
\begin{itemize}
\item harmonic in $U$ with respect to $X^{b, B}$ if
\begin{equation}\label{AV}
u(x) = \E_x[u(X^{b, B}_{\tau^b_{V}})],\qquad
x\in V,
\end{equation}
for every bounded open set $V$ with $\overline{V}\subset U$;

\item excessive with respective to $X^{b, B}$ if $u$ is
nonnegative and for every $t > 0$ and $x\in B$
\[
u(x) \geq \E_x[u(X^{b, B}_{t})]
\quad\mbox{and}\quad u(x) = \lim_{t \downarrow0}
\E_x[u(X^{b, B}_t)];
\]

\item a potential with respect to $X^{b, B}$ if it is excessive
with respect to $X^{b, B}$ and for every sequence $\{U_n\}_{n \ge1}$
of open sets with
$\overline{U}_n \subset U_{n+1}$ and $\bigcup_n U_n = B$,
\[
\lim_{n \to\infty}
\E_x[u(X^{b, B}_{\tau^b_{U_n}})] = 0,\qquad
\xi_B\mbox{-a.e. } x\in B;
\]

\item a pure potential with respect to $X^{b, B}$ if
it is a potential with respect to $X^{b, B}$ and

\[
\lim_{t \to\infty} \E_x[u(X^{b, B}_t)] = 0,\qquad
\xi_B \mbox{-a.e. } x\in B.
\]

\end{itemize}
\end{defn}

By Propositions \ref{P:G} and \ref{em} and an argument similar to that of \cite[Theorem 5.4]{CKS}, the following properties of the Green function $\overline{G}{}^b_B(x,y)$
of $X^{b, B}$ with respect to $\xi_B$ holds:
\begin{itemize}
\item[(H1)]
$\overline{G}{}^b_B(x,y) > 0$ for all $(x,y) \in B \times B$;
$\overline{G}{}^b_B(x,y) =\infty$ if and only if $x=y \in B$.
\item[(H2)]
For every $x \in B$, $\overline{G}{}^b_B(x, \cdot)$ and
$\overline{G}{}^b_B( \cdot,x)$ are extended continuous in~$B$.
\item[(H3)]
For every compact subset $K$ of $B$, $\int_K \overline{G}{}^b_B(x,y)
\xi_B(dy) < \infty$.

\item[(H4)] For every $y \in B$, $x\mapsto \overline{G}{}^b_B(x, y)$ is a pure potential with respect to $X^{b, B}$.
\item[(H5)] For every compact subset $K$ of $B$, $\int_K
\overline{G}{}^b_B(x,y) \xi_B(dx) < \infty$.
\end{itemize}
Using (H1)--(H5) and Proposition \ref{P:G}, we get from
\cite{L} that for each ball $B,$ $X^{b, B}$ has a Hunt process as a dual.
Denote by $P^{b, B}_t$ the semigroup of $X^{b, B}$.
\begin{thm} \label{dualpr}
For each ball $B,$ there exists a transient Hunt process $\wh X^{b, B} $ in $B$ such
that $\wh X^{b, B} $ is a strong dual\vadjust{\goodbreak} of $X^{b, B}$ with respect to
the measure $\xi_B$; that is, the density of the semigroup $\{ \wh
P^{b, B}_t\}_{t\geq0}$ of $\wh X^{b, B}$ with respect to $\xi_B(dy)$ is given by
$\wh p^{b, B}(t, x, y):=\overline{p}^{b, B}(t,y,x)$ and thus
\[
\int_B f(x) P^{b, B}_t g(x) \xi_B(dx)=\int_B g(x) \wh P^{b, B}_t f(x) \xi_B(dx)
\qquad\mbox{for all } f, g\in L^2(B, \xi_B).
\]
\end{thm}

Next we determine
a L\'evy system for $\wh X^{b, B}$. We first consider a L\'evy
system for $X^{b, B}$ with respect to the reference measure
$\xi_B(dx)$. Let $\overline J^B(x, y):=J(x, y)/h_B(y)$ for $x, y\in B.$ One can easily check that, if
$$\begin{aligned}
J^B(x,dy)&:=\overline J^B(x, y) \xi_B(dy)=J(x, y)\,dy \qquad\mbox{for
} (x,y) \in B \times B, \\
J^B(x,\partial)&:=\int_{B^c}J(x,y)\,dy \qquad\mbox{for }
x\in B,
\end{aligned}$$
and $H^B_t:=t$,\vspace*{1pt} then $(J^B, H^B)$ is
a L\'evy system for $X^{b, B}.$
It follows from \cite{G} that a L\'evy system $(\wh J^B, \wh H^B)$
for $\wh X^{b,B}$ satisfies $\wh H^B_t=t$ and
\[
\wh J^B(y,dx) \xi_B(dy) =J^B(x,dy)\xi_B (dx).
\]
 Thus we can take $\wh J^B(x, dy)=\overline J^B(y, x)\,\xi_B(dy)=\frac{J(y, x)h_B(y)}{h_B(x)} \,dy$.

\begin{lem}\label{L:4.5}
For each $R>0,$
 there is a constant  $C=C(d, R)>0$  such that for any $r\in (0, R)$ and $t>0,$
\begin{equation}\label{e:4.20'}
\P_x(\tau^b_{B(x,r)}\leq t)\leq  C   t/r^2  \quad \mbox{for every} \quad x\in \R^d.
\end{equation}
Consequently,  ${\bf (EP)} _{\leq, R, \gamma}$   holds  with  $\gamma=2.$
\end{lem}

\pf Let $R>0$ and $r\in (0,R).$ Note that $\frac{t}{r^2}>1$ for $t\in [r^2, \infty)$.
It suffices to prove \eqref{e:4.20'} for $t\in (0, r^2).$
By the strong Markov property of $X^b,$ for $x\in \R^d,$
\begin{eqnarray}\label{e:4.32}
\P_x(\tau^b_{B(x,r)}\leq t)
&=&\P_x(\tau^b_{B(x,r)}\leq t, X^b_{2t}\in B(x, r/2)^c)+\P_x(\tau^b_{B(x,r)}\leq t, X^b_{2t}\in B(x, r/2)) \nonumber \\
&\leq& \P_x(X^b_{2t}\in B(x, r/2)^c)+\sup_{s\leq t, z\in \R^d}\P_z(X^b_{2t-s}\in B(z, r/2)^c) \nonumber \\
&\leq & 2\sup_{s\leq 2t, z\in \R^d}\P_z(X^b_{s}\in B(z, r/2)^c).
\end{eqnarray}
By  \eqref{e:p},  there exist positive constants $c_1=c_1(d, R)$ and $c_2=c_2(d)$ such that for $t\in (0, r^2)$,
\begin{eqnarray} \label{e:4.34n}
&& \sup_{s\leq 2t, x\in\R^d}\P_x(X^b_s\in B(x, r/2)^c)
=\sup_{s\leq 2t, x\in\R^d}\int_{B(x, r/2)^c}  p^b(s, x, y)\,dy
\nonumber \\
&\leq & c_1\sup_{s\leq 2t, x\in\R^d}\int_{B(x, r/2)^c} \left(s^{-d/2}e^{-\frac{c_2|x-y|^2}{s}}+\dfrac{s}{|x-y|^{d+\beta}}\right)\,dy.
\end{eqnarray}
By using polar coordinate, for  every $x\in\R^d$ and $s\in (0, 2t)\subset (0, 2r^2),$
\begin{eqnarray} \label{e:4.35n}
&& \int_{B(x, r/2)^c} s^{-d/2}e^{-\frac{c_2|x-y|^2}{s}}\,dy
=\int_{r/(2\sqrt s)}^\infty e^{-c_2v^2}\,dv
= \int_{r^2/(4 s)}^\infty e^{-c_2u}\dfrac{1}{2\sqrt u}\,du \nonumber \\
&\leq & \dfrac{\sqrt s}{r}\int_{r^2/(4 s)}^\infty e^{-c_2u}\,du
= \dfrac{\sqrt s}{c_2r}e^{-c_2r^2/(4s)}\leq 2\dfrac{c_3}{c_2}\dfrac{s}{r^2} ,
\end{eqnarray}
where we used $e^{-c_2x}< c_3x^{-1/2}$  for $x>1/8$ in the last inequality.
Hence, by this inequality and \eqref{e:4.34n}, there exists $c_4=c_4(d, R)>0$ such that for $r\in (0, R)$ and $t\in (0, r^2),$
\begin{equation}\label{e:4.33}
\sup_{s\leq 2t, x\in \R^d}\P_x(X^b_s\in B(x, r/2)^c)\leq c_4 t/r^2,
\end{equation}
Hence by \eqref{e:4.32} and \eqref{e:4.33}, \eqref{e:4.20'} is obtained.
\qed

\smallskip

For each ball $B$ and each open subset $U\subset B,$ denote by $\wh\tau^{b, B}_U$ the first exit time of the dual process $\wh X^{b, B}$ from $U.$

\begin{lem}\label{L:4.6}
For each $R>0,$
 there is a constant  $C=C(d, R)>0$  such that for
  each ball $B$ with radius $R$,
  any $r\in (0, R/4)$ and $t>0,$
\begin{equation}\label{e:4.20}
\P_x\big(\wh\tau^{b, B}_{B(x,r)}< t\wedge\wh\tau^{b, B}_{2^{-1}B}\big)\leq  C  t/r^2
\quad \hbox{for }  x\in \tfrac{1}{2} B.
\end{equation}
Consequently, ${\bf (\wh{EP})} _{\leq, R, \gamma}$   holds   with  $\gamma=2.$
\end{lem}

\pf Note that $\frac{t}{r^2}>1$ for $t\in [r^2, \infty).$
It suffices to prove \eqref{e:4.20} for $t\in (0, r^2).$
For the simplicity of notation, let $\wh\tau:=\wh\tau^{b, B}_{2^{-1}B}.$  The following proof is similar to \cite[Lemma 2.4]{CS}.

 Let $x\in \tfrac{1}{2}B$ and $r\in (0, R/4)$. For any $s\in (0, r^2)$, by the strong Markov property of $\wh X^{b, B}$, we have
\begin{align}\label{e:EP-1}
		&\P_x\big( \wh\tau^{b, B}_{B(x,r)} < s \wedge \wh\tau, \, s/2 < \wh\tau  \big) \nonumber\\
		&\leq 	\P_x\big( \wh X^{b, B}_{s/2} \in \tfrac{1}{2}B \setminus B(x,r/2)  \big)  + 	\P_x\big( \wh X^{b, B}_{s/2} \in  \tfrac{1}{2}B\cap B(x,r/2) , \, \wh X^{b, B}_{\wh\tau^{b, B}_{B(x,r)}} \in \tfrac{1}{2}B \setminus B(x,r); \wh\tau^{b, B}_{B(x,r)}<s  \big)  \nonumber\\
&= 	\P_x\big( \wh X^{b, B}_{s/2} \in \tfrac{1}{2}B \setminus B(x,r/2)  \big)  + 	\P_x\big( \wh X^{b, B}_{s/2} \in  \tfrac{1}{2}B\cap B(x,r/2) , \, \wh X^{b, B}_{\wh\tau^{b, B}_{B(x,r)}} \in \tfrac{1}{2}B \setminus B(x,r); \wh\tau^{b, B}_{B(x,r)}<s/2  \big)  \nonumber\\
&\quad +	\P_x\big( \wh X^{b, B}_{s/2} \in  \tfrac{1}{2}B\cap B(x,r/2) , \, \wh X^{b, B}_{\wh\tau^{b, B}_{B(x,r)}} \in \tfrac{1}{2}B \setminus B(x,r); s/2<\wh\tau^{b, B}_{B(x,r)}<s  \big)  \nonumber\\
			&\leq 	\P_x\big( \wh X^{b, B}_{s/2} \in \tfrac{1}{2}B \setminus B(x,r/2)  \big)+ 	\P_x\big(   \wh X^{b, B}_{\wh\tau^{b, B}_{B(x,r)}} \in \tfrac{1}{2}B, \, \wh X^{b, B}_{s/2} \in  \tfrac{1}{2}B \setminus B( \wh X^{b, B}_{\wh\tau^{b, B}_{B(x,r)}},r/2) ; \wh\tau^{b, B}_{B(x,r)}<s/2 \big)     \nonumber\\
			&\quad  + 	\P_x\big(  \wh X^{b, B}_{s/2} \in \tfrac{1}{2}B\cap B(x, r/2), \, \wh X^{b, B}_{\wh\tau^{b, B}_{B(x,r)}} \in \tfrac{1}{2}B \setminus B(\wh X^{b, B}_{s/2},r/2) ; s/2<\wh\tau^{b, B}_{B(x,r)}<s  \big)  \nonumber\\
		&\leq 	3 \sup_{y \in \tfrac{1}{2}B, \, a \in (0,s/2]} \P_y\big(  \wh X^{b, B}_{a} \in  \tfrac{1}{2}B \setminus B(y,r/2)  \big).
\end{align}
By Theorem \ref{dualpr}, \eqref{e:h} and \eqref{e:p}, we have
\begin{align*}
&	\sup_{x\in \tfrac{1}{2}B, a\in (0, s/2)}\P_x\big( \wh X^{b, B}_a \in \tfrac{1}{2}B \setminus B(x,r/2)  \big) \\
& = \sup_{x\in \tfrac{1}{2}B, a\in (0, s/2)}\int_{\tfrac{1}{2}B \setminus B(x, r/2)} \wh p^{b, B}(a, x, y) \,\xi_B(dy) \\
& = \sup_{x\in \tfrac{1}{2}B, a\in (0, s/2)}\int_{\tfrac{1}{2}B \setminus B(x, r/2)} p^{b}_B(a,y,x) h_B(y)/h_B(x) \,dy \\
&\leq c_1\sup_{x\in\R^d, a\in (0, s/2)}\int_{\tfrac{1}{2}B \setminus B(x, r/2)}  p^b_B(a, y, x) \,dy\\
&\leq c_1\sup_{x\in\R^d, a\in (0, s/2)}\int_{\R^d \setminus B(x,r/2)}  p^b(a, y, x) \,dy\\
&\leq c_2\sup_{x\in\R^d, a\in (0, s/2)}\int_{B(x, r/2)^c} \left(a^{-d/2}e^{-\frac{c_3|x-y|^2}{2a}}+\dfrac{a}{|x-y|^{d+\beta}}\right)\,dy\\
&  \leq c_4s/r^2,
\end{align*}
where $c_k=c_k(d, R), k=1, \cdots, 4$ are independent of the center of $B$ and the last inequality is due to \eqref{e:4.35n} and $r\in (0, R/4).$
Combining this with \eqref{e:EP-1}, we obtain
\begin{align*}
\P_x\big( \wh\tau^{b, B}_{B(x,r)} < s \wedge \wh\tau, \, s/2 < \wh\tau  \big)  \le \frac{3c_4 s}{r^2} \quad \text{for any $s\in (0,r^2)$}.
\end{align*}
Using this, we conclude that for any $x\in\R^d$ and $t\in (0, r^2)$,
\begin{align*}
		\P_x\big( \wh\tau^{b, B}_{B(x,r)} < t \wedge \wh\tau  \big) &= 	\P_x\big( \wh\tau^{b, B}_{B(x,r)} < t < \wh\tau  \big)+  \sum_{n=1}^\infty 	\P_x\big( \wh\tau^{b, B}_{B(x,r)} <  \wh\tau, \, \wh\tau \in (2^{-n} t, 2^{1-n}t]\big)\\
		&\le 	\P_x\big( \wh\tau^{b, B}_{B(x,r)} < t \wedge \wh\tau, \, t/2<\wh\tau  \big)+  \sum_{n=2}^\infty 	\P_x\big( \wh\tau^{b, B}_{B(x,r)} <  2^{1-n}t \wedge \wh\tau, \, 2^{-n}t< \wh\tau \big)\\
		 &\le  \frac{3c_4t}{r^2}  \bigg( 1+  \sum_{n=2}^\infty 2^{1-n} \bigg) = \frac{9c_4t}{r^2}.
\end{align*}
\qed

\begin{proof}[Proof of   Proposition \ref{C6}.]
 It follows from \cite{CHXZ} that for each finite $t>0,$ the transition density of  $X^b$ is bounded from above by the transition density function of $\Delta+\Delta^{\beta/2}.$ As Hunt's hypothesis holds for $\Delta+\Delta^{\beta/2}$ (see e.g. \cite[Theorem 2.1]{HS1}), then by \cite[Theorem 1.1]{HS2}, it holds for $X^b.$ By  Theorem \ref{dualpr},
  for each ball $B$, the process $X^B$ is   in weak  duality to another
   Hunt process $\wh X^B$ with respect to  the reference measure $\xi_B$. Moreover, the $\alpha$-resolvent measures of $X^B$ and $\wh X^B$ are absolutely continuous with respect to $\xi_B.$ Hence assumption (A1)  holds.

 Next we prove assumption (A2') holds.
  By the condition \eqref{e:j}, for each $R>0$ and $r\in (0, R),$
$$J(x, B(x, r)^c)=\int_{B(x, r)^c} j(x, y-x)\,dy\leq \int_{B(x, r)^c} c_0|y-x|^{-(d+\beta)}\,dy\leq c_1r^{-\beta}\leq c_1 R^{2-\beta}r^{-2}.$$
Then ${\rm \bf (Jt)}_{\leq, \bar r, \gamma}$ holds with $\bar r=R$ and $\gamma=2.$
Let $B$ be a ball with radius $R>0$ in $\R^d.$
By \eqref{e:h} of Proposition \ref{em} and the condition \eqref{e:j}, there exist positive constants $c_k=c_k(d, R), k=2, 3$ such that for $x\in \tfrac{1}{2}B$ and $r\in (0, R/4),$
$$\begin{aligned}
&\wh J^B(x, (\tfrac{1}{2}B)\setminus B(x,r))=\int_{(\tfrac{1}{2}B)\setminus B(x,r)}\dfrac{J(y, x)h_B(y)}{h_B(x)}\,dy
\leq c_2\int_{(\tfrac{1}{2}B)\setminus B(x,r)}J(y, x)\,dy\\
\leq &\int_{(\tfrac{1}{2}B)\setminus B(x,r)}c_0c_2|x-y|^{-(d+\beta)}\,dy\leq \frac{c_3}{r^\beta}\leq c_3 R^{2-\beta} r^{-2}.\\
\end{aligned}$$
That is, ${\rm \bf (\wh{Jt})}_{\leq, \bar r, \gamma}$ holds  with $\bar r=R$ and $\gamma=2.$
 By Lemmas \ref{L:4.5} and \ref{L:4.6}, the conditions ${\bf (EP)} _{\leq, R, \gamma}$  and ${\bf (\wh{EP})} _{\leq, R, \gamma}$ hold with  $\gamma=2.$
 Hence by Proposition \ref{P1}, assumption (A2') holds for $X^b.$

 Let $B=B(x_0, R)$ be a ball centered at $x_0$ with radius $R>0.$
Note that $\overline{G}{}^b_B(x,y)$and $\overline J^B(x, y)$ are the Green function and the jump density of
of $X^{b, B}$ with respect to $\xi_B$ in $B$.
 By \eqref{e:h}, there exists $c_4=c_4(d, R)>0$ such that $\overline G^{b}_B(x, y)=G^b_B(x, y)/h_B(y)\leq c_4G^b_B(x, y)$ for  any $x, y\in \tfrac{1}{2}B\times \tfrac{1}{2}B\setminus {\rm diag}.$
Then by this together with Proposition \ref{P:G},  it is easy to check that assumption (A3') holds.
Note that  by \eqref{e:h} and \eqref{e:j}, for each  $R>0,$
\begin{equation}\label{e:4.41}\begin{aligned}
&\sup_{y\in B(x_0, R/2)\setminus B(x_0, R/16)} \overline J^B(x_0, y)=\sup_{y\in B(x_0, R/2)\setminus B(x_0, R/16)} J(x_0, y)/h_B(y)\\
&\leq c_5\sup_{y\in B(x_0, R/2)\setminus B(x_0, R/16)} J(x_0, y)\leq c_6\sup_{R/16\leq |z|< R/2} |z|^{-(d+\beta)}\leq c_6(R/16)^{-(d+\beta)}<\infty,
\end{aligned}\end{equation}
where $c_k, k=5, 6$ are independent of $x_0\in\R^d.$
 Combining this with the conditions \eqref{e:4.27} and \eqref{e:4J} of Proposition \ref{C6},  assumption (A4) holds.
Hence by the argument above, assumptions (A1), (A2')-(A3') and (A4) hold. Thus by Theorem \ref{T0}, the BHP \eqref{e:Xb} holds.

If we further assume \eqref{e:4.27} and \eqref{e:4J} hold uniformly for $x_0\in\R^d,$ note that \eqref{e:4.41} holds uniformly for $x_0\in\R^d,$ thus assumption (A4') holds.
As mentioned before,   for each ball $B=B(x_0, R)$ centered at $x_0$ with radius $R$, one has $G^b_B(x, y)\asymp G^\Delta_B(x, y)$, where $G^\Delta_B$  is the Green function of the Laplace operator on $B$ and the multiplicative constants are independent of the center $x_0$ of the ball $B(x_0, R).$
Hence for each $R>0,$
$$\begin{aligned}
&\sup_{x_0\in\R^d}\E_{x_0} \tau^b_{B(x_0, R)}=\sup_{x_0\in\R^d} \int_{B(x_0, R)} G^b_{B(x_0, R)}(x_0, y)\,dy\\
&\asymp \sup_{x_0\in\R^d} \int_{B(x_0, R)} G^\Delta_{B(x_0, R)}(x_0, y)\,dy
=\sup_{x_0\in\R^d}\E_{x_0} \tau^W_{B(x_0, R)} \leq c_7R^2,
\end{aligned}$$
 where   $\tau^W_{B(x_0, R)}$ is the first exit time of Brownian motion from $B(x_0, R)$.
So assumption (A5) holds. Hence assumptions (A1), (A2')-(A4') and (A5) hold. Then by Theorem \ref{T0},  BHP \eqref{e:Xb} holds  with $C=C(R)\geq 1$.
\end{proof}

\subsection{L\'evy processes with degenerate Gaussian component}

Let $d>m\geq 3$ be integers and $\alpha\in (0, 2).$ Let $X$ be  a  L\'evy process in $\R^d$ given by
$$X_t=(W_t, 0_{d-m})+Z_t,$$
where $0_{d-m}$ is the zero vector in $\R^{d-m},$ $W$ is a Brownian motion in $\R^m$ and $Z$ is a rotationally symmetric $\alpha$-stable process in $\R^d$   that  is independent of $W.$
It is known that the generator of $X$ is
$$\L f(x)=\sum_{i=1}^m \partial_{x_i}^2 f(x)+\lim_{\ee\rightarrow 0} \int_{\{z\in \R^d: |z|>\ee\}} (f(x+z)-f(x))\dfrac{\mathcal A(d, \alpha)}{|z|^{d+\alpha}}\,dz, \quad f\in C_c^\infty(\R^d),$$
 where  $\mathcal{A}(d, \alpha)$ is a positive normalizing constant so that the Fourier transform of $\widehat{\Delta^{\alpha/2}f}$ of $\Delta^{\alpha/2}f$ is $-|\xi|^\alpha \hat f(\xi).$

 \begin{prp}\label{C7}
 For each  $R>0,$
there exists a constant  $C=C(d, \alpha, m, R)>0$   such that for any open subset $D$ in $\R^d,$ $z_0\in \partial D$ and  for   any  nonnegative  regular harmonic functions $f$ and $g$
on $D\cap B(z_0, R)$   vanishing  on  $(D^c)^r\cap B(z_0, R),$
$$
f(x)g(y)\leq Cf(y)g(x)  \quad \hbox{for } x, y\in D\cap B(z_0, R/32).
$$
 \end{prp}

In the following, we give the proof of  Proposition \ref{C7}.

\begin{lem}\label{L:4.18}
For each $R>0,$ ${\bf (ED)}^s _{\leq, R}$ and  ${\bf (\wh{ED})}^s _{\leq, R}$   hold for $X$.
That is, assumption (A2') holds for $X$.
\end{lem}

\pf Let $\varphi\in C_c^\infty(\R^d)$ be such that $0\leq \varphi\leq 1, \varphi(y)=0$
if $|y|\leq  1/64$ or $|y|\geq 1$ and $\varphi(y)=1$ if $1/32\leq |y|\leq 1/2.$
For each $x_0\in\R^d$ and $r>0,$ define $\varphi_{x_0,r}(y):=\varphi((y-x_0)/r).$
Note that for $x\in\R^d,$
$$\begin{aligned}
&\lim_{\ee\rightarrow 0} \int_{z\in \R^d: |z|>\ee} (\varphi_{x_0,r}(x+z)-\varphi_{x_0,r}(x))\dfrac{\mathcal A(d, \alpha)}{|z|^{d+\alpha}}\,dz\\
=&\int_{\{z\in\R^d: |z|\leq  r\}} \left(\varphi_{x_0,r}(x+z)-\varphi_{x_0,r}(x)-\nabla \varphi_{x_0,r} (x) \cdot z \right) \dfrac{\mathcal A(d, \alpha)}{|z|^{d+\alpha}}\,dz\\
&+\int_{\{z\in\R^d: |z|> r\} } (\varphi_{x_0,r}(x+z)-\varphi_{x_0,r}(x)) \dfrac{\mathcal A(d, \alpha)}{|z|^{d+\alpha}}\,dz.
\end{aligned}$$
Since
$\frac{\partial^2}{\partial y_iy_j} \varphi_{x_0,r}(y)=r^{-2}\frac{\partial^2 \varphi }{\partial y_iy_j}  ((y-x_0)/r) $  and $\nabla   \varphi_{x_0,r} (y) = r^{-1} \nabla \varphi ((y-x_0)/r)$, we have for $x_0\in\R^d$ and $r>0,$
\begin{equation}\label{e:4.42}\begin{aligned}
&\sup_{y\in\R^d}|\L \varphi_{x_0,r}(y)|\\
\leq &\sup_{y\in\R^d}\left|\sum_{i=1}^m  \partial^2_{y_i} \varphi_{x_0,r}(y)\right|+\sup_{y\in\R^d}\int_{|z|\leq  r } \left|\varphi_{x_0,r}(y+z)-\varphi_{x_0,r}(y)-\nabla \varphi_{x_0,r} (y) \cdot z\right| \dfrac{\mathcal A(d, \alpha)}{|z|^{d+\alpha}}\,dz\\
&+\sup_{y\in\R^d}\int_{|z|> r } |\varphi_{x_0,r}(y+z)-\varphi_{x_0,r}(y)| \dfrac{\mathcal A(d, \alpha)}{|z|^{d+\alpha}}\,dz\\
\leq &r^{-2}\left( \| D^2 \varphi\|_\infty +\| D^2 \varphi\|_\infty \ \int_{|z|\leq  r } |z|^2\dfrac{\mathcal A(d, \alpha)}{|z|^{d+\alpha}}\,dz\right)
+2  \int_{|z|> r } \dfrac{\mathcal A(d, \alpha)}{|z|^{d+\alpha}}\,dz\\
\leq & c_1(r^{-2}+r^{-\alpha}),
\end{aligned}\end{equation}
where $c_1=c_1(d,\alpha)>0$ independent of $x_0\in\R^d$ and $r>0$ and $\| D^2 \varphi\|_\infty:=\sup_{y\in\R^d}\sum_{i, j=1}^d|\frac{\partial^2}{\partial y_iy_j} \varphi(y)|.$
Thus  by Ito's formula, for each $x_0\in\R^d$ and $R>0,$ open set $U\subset B(x_0, R/2)$ and $x\in B_U(x_0, R/64),$
$$\begin{aligned}&\P_x(X_{\tau_{B_U(x_0, R/32)}}\in U)
 \leq \E_x \varphi_{x_0 ,R}(X_{\tau_{B_U(x_0, R/32)}})\\
 &\leq \E_x\int_0^{\tau_{B_U(x_0, R/32)}} |\L \varphi_{x_0,R}(X_s)|\,ds
 \leq c_1(R^{-2}+R^{-\alpha})\E_x\tau_{B_U(x_0, R/32)}.
 \end{aligned}$$
 Hence   for each $R>0,$ ${\bf (ED)}^s _{\leq, R}$   holds.

 Let $\psi\in C_c^\infty(\R^d)$ be such that $0\leq \psi\leq 1,$  $\psi(y)=0$
if $|y|\geq 3/8$ and $\psi(y)=1$ if $|y|\leq 1/4.$
For each $x_0\in\R^d$ and $R>0,$ define $\psi_{x_0,R}(y):=\psi((y-x_0)/R).$
By an argument similar to that of \eqref{e:4.42}, there exists $c_2=c_2(d, \alpha)>0$ such that for any $x_0\in\R^d$ and $R>0,$
$$\sup_{y\in\R^d}|\L \psi_{x_0,R}(y)|\leq c_2(R^{-2}+R^{-\alpha}).$$
Then by this inequality and Ito's formula, for each $x_0\in\R^d$ and $R>0,$ any open set $U\subset B(x_0, R/2)$ and  $x\in U\cap A(x_0, 3R/8, R/2),$
$$ \P_x  \big( X_{\tau_{U\cap A(x_0, R/4, R/2)}} \in U \big)\leq \E_x \psi_{x_0 ,R}(X_{\tau_{U\cap A(x_0, R/4, R/2)}})
\leq c_2(R^{-2}+R^{-\alpha})\E_x\tau_{U\cap A(x_0, R/4, R/2)}.$$
Then by combining with Remark \ref{R:1.1}, ${\bf (\wh{ED})}^s _{\leq, R}$   holds.
Hence,  assumption (A2') holds.
\qed

\medskip

Since $X$ is rotationally symmetric, we denote by $p(t, x-y)=p(t, x, y)$ the transition density function of $X$ in $\R^d.$
Denote by $p^W(t, u-v)=p^W(t, u, v)$ and $p^Z(t, x-y)=p^Z(t, x, y)$ the transition density functions of $W$ in $\R^m$ and $Z$ in $\R^d$ respectively.
For each $x=(x_1, \cdots, x_d)\in\R^d,$  let $x^{(m)}:=(x_1, \cdots, x_m)$ and $x_{(d-m)}:=(x_{m+1}, \cdots, x_d).$

\begin{lem}\label{L:4.19}
There exists $C=C(d, \alpha)>0$ such that for any $x\in\R^d$ and $t>0,$
$$
p(t, x)\leq C\left( p^W(4t, x^{(m)}) \wedge p^Z(t, (0, x_{(d-m)}))\right)+Cp^Z(t, x).
$$
\end{lem}

\pf Let $k$ be an integer and $0_k$ be the zero vector in $\R^k.$ For each $u\in\R^m,$ we write $\tilde u:=(u, 0_{d-m})\in\R^d.$ Since $W$ and $Z$ are independent, we have for $t>0$ and $x\in\R^d,$
\begin{equation}\label{e:4.34}\begin{aligned}
p(t, x)&= \int_{\R^m} p^Z(t, x-\tilde u) p^W(t, u)\,du\\
&=  \int_{u\in\R^m: |u-x^{(m)}|<|x^{(m)}|/2} p^Z(t, x-\tilde u) p^W(t, u)\,du\\
&\quad +\int_{u\in\R^m: |u-x^{(m)}|\geq |x^{(m)}|/2} p^Z(t, x-\tilde u) p^W(t, u)\,du\\
&:=I_1(x)+I_2(x).
\end{aligned}\end{equation}
Recall that $p^W(t, u)=(2\pi t)^{-m/2} e^{-|u|^2/2t}$ for $t>0$ and $u\in \R^m.$
Then we have
$$\begin{aligned}
I_1(x)&= \int_{u\in\R^m: |u-x^{(m)}|<|x^{(m)}|/2} p^Z(t, x-\tilde u) (2\pi t)^{-m/2} e^{-|u|^2/2t}\,du\\
&\leq (2\pi t)^{-m/2} e^{-|x^{(m)}|^2/8t}\int_{u\in\R^m: |u-x^{(m)}|<|x^{(m)}|/2} p^Z(t, x-\tilde u)\,du\\
&\leq (2\pi t)^{-m/2}  e^{-|x^{(m)}|^2/8t}=2^mp^W(4t, x^{(m)}).
\end{aligned}$$
On the other hand, it is known that there exists $c_1>1$ such that $c_1^{-1}(t^{-d/\alpha}\wedge |y|^{-(d+\alpha)})\leq p^Z(t, y)\leq c_1(t^{-d/\alpha}\wedge |y|^{-(d+\alpha)})$ for $t>0$ and $y\in\R^d,$ we have
$$\begin{aligned}
I_1(x)&= \int_{u\in\R^m: |u-x^{(m)}|<|x^{(m)}|/2} p^Z(t, x-\tilde u) p^W(t, u)\,du\\
&\leq c_1\int_{u\in\R^m: |u-x^{(m)}|<|x^{(m)}|/2} (t^{-d/\alpha}\wedge |x-\tilde u|^{-(d+\alpha)}) p^W(t, u)\,du\\
&\leq c_1\left(t^{-d/\alpha}\wedge |(0_m, x_{(d-m)})|^{-(d+\alpha)}\right) \int_{u\in\R^m: |u-x^{(m)}|<|x^{(m)}|/2} p^W(t, u)\,du\\
&\leq c_1\left(t^{-d/\alpha}\wedge |(0_m, x_{(d-m)})|^{-(d+\alpha)}\right)\leq c_1^2p^Z(t, (0_m, x_{(d-m)})) .
\end{aligned}$$
Hence
\begin{equation}\label{e:4.35}
I_1(x)\leq (2^m+c_1^2) \left( p^W(4t, x^{(m)}) \wedge p^Z(t, (0_m, x_{(n-m)}))\right).
\end{equation}
For the second item,
\begin{equation}\label{e:4.36}\begin{aligned}
I_2(x)&\leq c_1\int_{u\in\R^m: |u-x^{(m)}|\geq |x^{(m)}|/2} \left(t^{-d/\alpha} \wedge t|x-\tilde u|^{-(d+\alpha)}\right) p^W(t, u)\,du\\
&\leq c_12^{d+\alpha}\left(t^{-d/\alpha} \wedge t|x|^{-(d+\alpha)}\right)\int_{u\in\R^m: |u-x^{(m)}|\geq |x^{(m)}|/2} p^W(t, u)\,du\\
&\leq c_12^{d+\alpha}\left(t^{-d/\alpha} \wedge t|x|^{-(d+\alpha)}\right)\leq c_1^22^{d+\alpha}p^Z(t, x),
\end{aligned}\end{equation}
where the second inequality holds as  $|x-\tilde u|\geq |x|/2$ for $u\in\R^m$ with $|u-x^{(m)}|\geq |x^{(m)}|/2.$
Hence, by \eqref{e:4.34}, \eqref{e:4.35} and \eqref{e:4.36}, the proof is complete.
\qed

\medskip

Denote by $G(x-y):=G(x, y)=\int_0^\infty p(t, x-y)\,dt$ the Green  function of $X$ in $\R^d.$
Denote by $G^W(\cdot)$ and $G^Z(\cdot)$ the Green function of $W$ in $\R^m$ and the Green function of $Z$ in $\R^d$, respectively.

\begin{lem}\label{L:4.20}
There exists $C>0$ such that for $x\in\R^d\setminus \{0\},$
\begin{equation}\label{e:4.39}
G(x)\leq C\left( |x^{(m)}|^{2-m}\wedge | x_{(n-m)}|^{\alpha-d}\right)+C |x|^{\alpha-d}.
\end{equation}
Consequently, assumption (A3') holds for $X$.
\end{lem}

\pf By Lemma \ref{L:4.19}, there exists $c_1=c_1(d, \alpha)>0$ such that for $x\in\R^d\setminus\{0\},$
$$\begin{aligned}
G(x)&=\int_0^\infty p(t, x)\,dt\\
&\leq c_1\left(\int_0^\infty p^W(4t, x^{(m)})\,dt\wedge \int_0^\infty p^Z(t, (0, x_{(n-m)}))\,dt\right)+ c_1\int_0^\infty p^Z(t, x)\,dt\\
&= c_1\left( G^W(x^{(m)})\wedge G^Z((0, x_{(n-m)})) \right)+c_1G^Z(x)\\
&\leq c_2\left(|x^{(m)}|^{2-m}\wedge | x_{(n-m)}|^{\alpha-d}\right)+c_2 |x|^{\alpha-d},
\end{aligned}$$
where $|\cdot|$ is the length of the vector and the last inequality is due to that $G^W(x^{(m)})\leq c|x^{(m)}|^{2-m}$ and $G^Z((0, x_{(n-m)}))\leq c| x_{(n-m)}|^{\alpha-d}$ for some $c>0.$ Hence \eqref{e:4.39} holds.

In the following, we prove assumption (A3') holds.
For each ball $B,$ denote by $G_B(x, y)$ the Green function of the subprocess of $X$ in $B$ with respect to the Lebesgue measure.
By \eqref{e:4.39}, for each $R>0,$
 $$\begin{aligned}
 &\sup_{x_0\in\R^d}\sup_{x\in B(x_0, R/16)}\sup_{y\in B(x_0, R/2)\setminus B(x_0, R/8)}G_{B(x_0, R)}(x, y)\\
 \leq &\sup_{R/16\leq |u|<R}G(u)\\
 \leq &c_2\sup_{R/16\leq |u|<R}\left( \left(|u^{(m)}|^{2-m}\wedge |u_{(n-m)}|^{\alpha-d}\right) + |u|^{\alpha-d}\right).
 \end{aligned}$$
Note that for $R/16\leq |u|<R,$  one has either $|u^{(m)}|>R/32$ or $|u_{(n-m)}|>R/32.$
Hence,
$$\sup_{x_0\in\R^d}\sup_{x\in B(x_0, R/16)}\sup_{y\in B(x_0, R/2)\setminus B(x_0, R/8)}G_{B(x_0, R)}(x, y)
\leq c_2(R/32)^{2-m}+c_2(R/32)^{\alpha-d}+c_2(R/16)^{\alpha-d}.
$$
Hence assumption (A3') holds.
\qed

\medskip

\begin{proof}[Proof of  Proposition \ref{C7}.]
Since $X$ is symmetric and has transition density function $p(t, x, y)$ with respect to the Lebesgue measure,  then for each ball $B$,   the $\alpha$-resolvent measure $G^{\alpha, B}(x, \cdot)$ of the subprocess $X^B$ is absolutely continuous with respect to the Lebesgue measure.
By \cite[Theorem 2.1]{HS1} (which is due to \cite{F,Ka}), Hunt's hypothesis holds for symmetric L\'evy process. Thus  assumption (A1) holds for $X$.
By Lemmas \ref{L:4.18} and \ref{L:4.20}, assumptions (A2') and  (A3') hold.
It is easy to see that the jump density function $J(x, y)=\mathcal{A}(d, \alpha)|x-y|^{-(d+\alpha)}$ with respect to the Lebesgue measure satisfies assumption (A4').
By Lemma \ref{L:4.20},  there exist $c_k, k=1, 2>0$ such that for $u\in\R^d$,
$$G(u)\leq c_1\left( |u^{(m)}|^{2-m}\wedge | u_{(n-m)}|^{\alpha-d}\right)+c_1 |u|^{\alpha-d}\leq c_2 (|u|^{2-m}+|u|^{\alpha-d}),$$
where the last inequality holds as  one has either
$|u^{(m)}|>|u|/2$ or $| u_{(n-m)}|>|u|/2$.
For each ball $B,$ denote by $\tau_B$ the first exit time of $X$ from $B.$
Hence for each $r>0,$
$$\begin{aligned}
& \sup_{x\in\R^d}\E_x \tau_{B(x, r)}\leq \sup_{x\in\R^d}\int_{B(x, r)} G(x-y)\,dy\\
& \leq \int_{u\in\R^d: |u|\leq 2r} G(u)\,du\leq c_2\int_{u\in\R^d: |u|\leq 2r} (|u|^{\alpha-d}+|u|^{2-m})\,du\leq c_3 (r^\alpha+r^{2+d-m})<\infty.
\end{aligned}$$
Then assumption (A5) holds.
Consequently, assumptions (A1), (A2')-(A4') and (A5)
are satisfied. Hence,   Proposition \ref{C7} holds by Theorem \ref{T0}.
 \end{proof}

\end{document}